\documentclass[11pt]{article}%

\usepackage{comment}

\usepackage[babel,german=quotes]{csquotes}

\usepackage[utf8]{inputenc}

\usepackage{color}

\usepackage{nicefrac}

\usepackage{amssymb}

\usepackage{amsmath}

\usepackage{amsmath}

\usepackage{amsthm}

\usepackage{xcolor}

\usepackage{amsfonts}

\usepackage{graphicx}

\usepackage{enumerate}

\usepackage{dsfont}

\usepackage{hyperref}

\usepackage{enumitem}

\hypersetup{pdfstartview={FitH},
colorlinks=true,
linkcolor=black,
urlcolor=black,
citecolor=black,
bookmarksopen,
bookmarksopenlevel={1},
bookmarksnumbered}

\usepackage[colorinlistoftodos,textsize=tiny]{todonotes}
\newcommand{\Comments}{1}
\newcommand{\mynote}[2]{\ifnum\Comments=1\textcolor{#1}{#2}\fi}
\newcommand{\mytodo}[2]{\ifnum\Comments=1%
	\todo[linecolor=#1!80!black,backgroundcolor=#1,bordercolor=#1!80!black]{#2}\fi}

\newcommand{\hajo}[1]{\mytodo{green!20!white}{HH: #1}}

\ifnum\Comments=1               % fix margins for todonotes
\fi

\providecommand{\U}[1]{\protect\rule{.1in}{.1in}}

\numberwithin{equation}{section}

\newtheorem{theorem}{Theorem}

\newtheorem{corollary}[theorem]{Corollary}

\newtheorem{proposition}[theorem]{Proposition}

\newtheorem{lemma}[theorem]{Lemma}

\theoremstyle{definition}

\newtheorem{remark}[theorem]{Remark}

\newtheorem{ass}{Assumption}

\DeclareMathOperator{\E}{\mathbb{E}}

\DeclareMathOperator{\var}{\mathbf{var}}

\DeclareMathOperator{\cov}{\mathbf{cov}}

\newcommand{\dd}{\mathrm{d}}

\newcommand{\cG}{\mathcal{G}}

\newcommand{\bbe}{ \mathbf{e}}

\newcommand{\bbk}{\mathbf{k}}

\newcommand{\R}{\mathbb{R}}

\newcommand{\cI}{\mathcal{I}}

\newcommand{\cK}{\mathcal{K}}

\newcommand{\cF}{\mathcal{F}}

\newcommand{\N}{\mathbb{N}}

\newcommand{\op}{\mathrm{op}}

\renewcommand{\Pr}{\mathbb{P}}

\usepackage[style=authoryear-comp, sorting=nyt, natbib=true, dashed=false,
            backend=bibtex, maxcitenames=3, maxbibnames=99]{biblatex}
\AtEveryBibitem{\clearlist{language}\clearfield{issn}}

\bibliography{literatur}

\begin{document}

\title{Local polynomial density ratio estimation }

\author{Hajo Holzmann\footnote{Corresponding author. Prof.~Dr.~Hajo Holzmann, Fachbereich Mathematik und Informatik, Philipps-Universit\"at Marburg, Hans-Meerweinstr.~6, 35043 Marburg, Germany}\\
\small{Fachbereich Mathematik und Informatik}  \\
\small{Philipps-Universit\"at Marburg} \\
\small{holzmann@mathematik.uni-marburg.de}
\and
Alexander Meister \\
\small{Institut f\"ur Mathematik}  \\
\small{Universit\"at Rostock} \\
\small{alexander.meister@uni-rostock.de}}

\date{\today }

\maketitle

\begin{abstract}
We  propose a novel local-polynomial estimator of the ratio $r=f/g$ of two $d$-dimensional densities $f$ and $g$, from which independent samples are available. The estimator is shown to achieve pointwise minimax optimal rates over Hölder classes of arbitrary smoothness index without additional logarithmic factors, and with smoothness being only assumed of $r$ but not of $f$ nor $g$. Our analysis remains  valid for points on the boundary of the support of the distribution associated to $g$ under a mild geometric assumption on the (unknown) boundary. We also derive a concentration inequality for the estimator, which can be useful in applications to classification, and give a rate in the supremum norm, where the supremum is also taken over boundary support points. The smoothness class over which the rate is obtained is sufficiently large that the individual densities $f$ and $g$ cannot be consistently estimated uniformly over this class. Direct estimators of the partial derivatives of $r$ together with rates of convergence are also provided. 

We also obtain asymptotic normality of the estimators under quite general assumptions and again including boundary points, with a consistent variance estimator allowing for data-driven studentization.
Moreover, we show how our estimator can be used to estimate the Kullback-Leibler information, in a construction which additionally uses debiasing. We provide the parametric rate together with asymptotic normality under sufficient smoothness of $r$ relative to the dimension $d$.
\end{abstract}

\vspace{3mm}


\section{Introduction}
%%%%%%%%%%%%%%%%%%%%%%

The ratio $r(x) := f(x)/g(x)$ of two Lebesgue densities $f$ and $g$ on $\mathbb{R}^d$ is an object of 
central importance in statistics and machine learning. Hence, estimating $r$ from two independent samples, $X_1,\ldots,X_n$ i.i.d.~with density $f$ and $Y_1,\ldots,Y_m$ i.i.d.~with density $g$, has been quite intensively studied, see \citet{SugiyamaSuzukiKanamori2012} for an overview. 
Here, areas of application include transfer learning under covariate shift \citep{shimodaira2000improving, gretton2009covariate, portier2023scalable, holzmann2025multivariate}, where the quotient of the marginal densities for source and target covariates  is required as a weight when calibrating loss functions from source to target; as well as importance sampling and simulator-based models \citep{delyon2016integral, thomas2022likelihood} and causal inference \citep{hirano2003efficient, lin2023estimation}. In information theory the density ratio is required for the estimation of  the Kullback-Leibler divergence and more general $f$-divergences \citep{BS23,  wang2009divergence, nguyen2010estimating, BSY19, poczos2011estimation}. The density ratio is also intrinsically related to the quotient of two intensity functions of spatial point processes, called the relative risk surface in epidemiology \citep{Bithell1990, Bithell1991, KelsallDiggle1995a, KelsallDiggle1995b, KelsallDiggle1998, HazeltonDavies2009, DaviesHazelton2010}.

\smallskip

In this paper we  set-up a novel local-polynomial estimator of the density ratio. Benefits of our method compared to previously suggested methods include: pointwise minimax optimal rates over Hölder classes of arbitrary smoothness index without additional logarithmic factors, and with smoothness being only assumed of $r$ but not of $f$ nor $g$; validity at boundary points under mild geometric assumptions on the (unknown) boundary; asymptotic normality under quite general assumptions and again including boundary points, with a consistent variance estimator allowing for data-driven studentization; and direct estimates of the partial derivatives of $r$.   

\smallskip

Let us give an overview over previously suggested methods to estimate $r$. A first, straightforward approach is to use the ratio of two density estimators. \citet{KelsallDiggle1995a, kpotufe2017lipschitz} use kernel estimators, the latter with radial kernel, and highlight the role of equal bandwidths, in which case this estimator arises as a special case of our general local polynomial density ratio estimator with order $0$. Similar approaches are the ratio of two k-nearest neighbor density estimators \citep{noshad2017direct} as well as matching estimators \citep{lin2023estimation}. These approaches can only make use of Lipschitz-smoothness at the most, with the matching estimators of \citet{lin2023estimation} and the kNN estimator of \citet{noshad2017direct} also requiring the two individual densities to be smooth. 

A second line of estimators arises from the insight that the density ratio can be realized as minimizer of an  $L_2$-contrast, giving rise to least-squares importance fitting e.g.~with basis functions  \citep{KanamoriHidoSugiyama2009,  kanamori2012statistical} or more recently feedforward neural networks \citep{XYH25}. More generally, the density ratio also arises as minimizer of suitable information-theoretic contrasts \citep{sugiyama2008direct, nguyen2010estimating, menon2016linking, zellinger2023adaptive}. In this line, most results are on global rates, e.g.~in $L_2$, often with additional logarithmic factors, while pointwise rates and asymptotic normality are not investigated.  

A third line is the regression-based or probabilistic classification-based approach, which uses artificially labeled samples (label $1$ say for the $X$-sample, and $0$ for the $Y$-sample), and then estimates the regression function $\zeta = \pi f / (\pi f + (1-\pi)\,  g)$, $\pi = n/(n+m)$, for example using a regression kNN estimator, after which one can transform back to
%
%\begin{equation*}%\label{eq:densratioregress}
 $r(x) = m\, \zeta(x)/(n\, (1-\zeta(x)))$, 
%\end{equation*} 
see \citep{menon2016linking} for relations between loss function for regression and density ratio estimation. 
 In the semiparametric density ratio model, or exponential tilt model, $\log r$ is assumed to be linear in finitely many parameters, which is equivalent to a logistic regression model  \citep{ prentice1979logistic, qin1998inferences, cheng2004semiparametric}. \citet{fan1995local, FFG98} extend this via a local likelihood approach to local logistic regression, also using local polynomials. However, a systematic pointwise analysis of rates and asymptotic normality does not seem to be available in the regression-based approach. Further, care needs to be taken of the relative sizes $n/(n+m)$ and $m/(n+m)$ of the samples, and also of the indirect nature of the estimate.

On the methodological side our approach is evidently related to local polynomial regression estimation \citep{stone1980optimal, FG96, audibert2007fast}, to the construction of 
multivariate boundary kernels in density estimation \citep{muller1999multivariate, Bertin2025adaptivedensity} as well as to local polynomial univariate density estimation and conditional density estimation \citep{cattaneo2020simple, cattaneo2024boundary}. \citet{bordino2025nonparametric} propose statistical tests for the form of the density ratio. 

Let us give an overview of the contents of the paper. In Section \ref{sec:densityratioest} we formally introduce the local polynomial density ratio estimator as well as the resulting estimator of the derivatives, and relate to both local polynomial estimation as well as to least squares importance fitting. 
Section \ref{sec:rates} contains the theoretical analysis of the rates of convergence. Assuming Hölder smoothness for $r$ but not necessarily individual smoothness of $f$ and $g$, as well as a mild additional boundary condition for boundary points we show that our estimator achieves minimax optimal rates of convergence, pointwise and in $L_2$. We also derive a concentration inequality for the estimator, which can be useful in applications to classification, and give a rate in the supremum norm, where the supremum is also taken over boundary support points. Rates for derivative estimation are also provided. Finally, we show that  our smoothness class is sufficiently large that the individual densities $f$ and $g$ cannot be consistently estimated uniformly over this class.  
Section \ref{sec:asympnorm} turns to the asymptotic distribution of the estimator. We show asymptotic normality, including at  boundary points of the support, and also propose a consistent estimator for the variance which can be used for studentization. Joint asymptotic normality of several derivative estimators is also obtained. 
In Section \ref{sec:KLest} we show how our estimator can be used to estimate the Kullback-Leibler information, in a construction which additionally uses debiasing. We provide the parametric rate together with asymptotic normality under sufficient smoothness of $r$ relative to the dimension $d$.  

Let us conclude the introduction by introducing some notation. For ${\bf k} = (k_1,\ldots,k_d) \in \N_0^d$ we let $|{\bf k}| = k_1 + \ldots + k_d$, ${\bf k}! = k_1!\, \cdot \ldots \cdot  k_d!$, and for $z = (z_1, \ldots, z_d)\in \R^d$ we let   
$z^{\bf k} = z_1^{k_1}\cdot \ldots \cdot z_d^{k_d}.$ Further let
$${\cal K}_l := \{{\bf k} = (k_1,\ldots,k_d)\in \mathbb{N}_0^d \mid |{\bf k}| \leq l\}, \qquad l \in \N_0.$$ 
For $L:\R^d \to \R$ and $h>0$  we let $L_h(z)=L(z/h)/h^d$.
We let $\| \cdot \|_{\mathrm{op}}$ denote the spectral norm of a symmetric matrix, $\| \cdot \|_F$ the Frobenius norm, and $\|\cdot \|_2$ or simply $\|\cdot \|$ the Euclidean norm of a vector.

%\section{Methodology}
%%%%%%%%%%%%%%%%%%%%%

%

\section{The local polynomial density ratio estimator}\label{sec:densityratioest}
%%%%%%%%%%%%%%%%%%%%%%%%%%%%%%%%%%%%

We describe our approach to construct a local polynomial estimator $\hat{r}(x)$ of $r(x)$ at a given point $x$, which uses similar techniques as in local polynomial regression \citep{FG96, audibert2007fast} and in local polynomial univariate density estimation and conditional density estimation \citep{cattaneo2020simple, cattaneo2024boundary}. %In particular, our results will apply to boundary points of $g$. 

To set the stage for local estimation of $r$ at $x$, for $\rho_1>0$ let ${\cal U} := {\cal U}_{\rho_1}(x) := \{z\in \mathbb{R}^d \mid \|x-z\|<\rho_1\}$. Assume that for suitable $\rho_1>0$,  
$f\,\dd\lambda^d \ll g\,\dd\lambda^d$ on ${\cal U}$, where $\lambda^d$ denotes the $d$-dimensional Lebesgue-Borel measure, and that a continuous 
version $r$ of the density of $f \,\dd\lambda^d/g \,\dd\lambda^d$ on ${\cal U}$ can be chosen, in particular, $f = r\,g$ $\lambda^d$-almost everywhere
on $\mathcal U$. At points $x$ on the boundary of the support of $g$ this requires a continuous extension of the density of $f \,\dd\lambda^d/g \,\dd\lambda^d$.

We fix the order $l \in \N_0$, which is suppressed in the notation of the estimator. 
Let 
\begin{equation} \label{eq:Deltahatapp}
\hat{r}(x;h)  \, = \,\frac1{n\, h^d}\, \sum_{j=1}^n \hat{K}\Big(\frac{x-X_j}{h};x,h\Big) \,,
\end{equation}
for a bandwidth parameter $h>0$, and with the empirical kernel function $\hat{K}(z; x,h)$ depending on $Y_1,\ldots,Y_m$, as well as on $x$ and $h$. For simplicity we often write $\hat{r}(x;h) = \hat r(x)$ and $\hat{K}(z; x,h) = \hat{K}(z)$, so that the estimator is given as
$\hat r(x) \, := \, \frac1n \sum_{j=1}^n \hat{K}_h(x-X_j). $
The empirical kernel $\hat K$ is defined as 
\begin{equation}\label{eq:empkernel}
 \hat{K}(z;x,h) \, := \,  K(z)\, \sum_{{\bf k}\in {\cal K}_l} \hat{\alpha}_{\bf k}(x;h) \cdot z^{\bf k} \,. \qquad 
 \end{equation}
Here the kernel function $K$ is assumed to be continuous and supported on the closed Euclidean  ball $B =B_1(0)$ with the center $0$ and the radius $1$, to be upper-bounded by $1$ and to satisfy $K(z)>0$ for all $z\in \mathbb{R}^d$ with $\|z\|<1$. 
The vector of weights $\hat{\alpha}$ is given by
\begin{equation} \label{eq:hatalpha} \hat{\alpha}(x;h)  \, := \,  \big(\hat{S}(x;h)\big)^{-1} {\bf e}_0\,, \end{equation}
with the $\#{\cal K}_l$-dimensional unit vector
${\bf e}_0 := (1,0,\ldots,0)^\top$, and the $\#{\cal K}_l\times\#{\cal K}_l$ Gram-matrix $\hat{S}(x;h)  = \{\hat{S}_{{\bf k},{\bf k}'}(x;h)\}_{{\bf k},{\bf k}'}$ with 
\begin{equation}\label{eq:gramemp}
\hat{S}_{{\bf k},{\bf k}'}(x;h)  \, := \,  \frac1m \sum_{j=1}^m K_h(x-Y_j) \cdot   \Big(\frac{x-Y_j}h\Big)^{\bf k+\bf k'}\,, \qquad {\bf k},{\bf k}' \in {\cal K}_l\,. 
\end{equation}
Again we also omit $x$ and $h$ in the notation and write $\hat{\alpha}$, $\hat{S} $ and $ \hat{S}_{{\bf k},{\bf k}'}$. 

Further, we note that setting
\begin{equation}\label{eq:estdensrat}
 \hat V(x;h) \, := \, \hat V \, := \,  \big(\hat V_{\bf k}(x;h)\big)_{{\bf k} \in {\cal K}_l},\qquad \hat V_{\bf k}(x;h) = \frac1n \sum_{j=1}^n {K}_h(x-X_j)\, \Big(\frac{x-X_{j}}h\Big)^{\bf k}, \quad {\bf k} \in {{\cal K}_l}, 
 \end{equation}
we may write
\begin{equation}\label{eq:esmatvec}
 \hat{r}(x) = {\bf e}_0^\top\, \hat{S}^{-1}\, \hat V.
\end{equation}
Finally, if $\hat S$ is not invertible we let $\hat r(x) = 1$. 

\smallskip

\begin{remark}[Population version]
The estimator is built from the empirical quantities $\hat S$ from \eqref{eq:gramemp} and $\hat V$ from \eqref{eq:estdensrat}. 
Consider the expectations
\begin{equation}\label{eq:Sdef}
 S = \E[\hat S] =  \{S_{{\bf k},{\bf k}'}\}_{{\bf k},{\bf k}'\in {{\cal K}_l}},\qquad \text{so}\quad S_{{\bf k},{\bf k}'} \, = \, \int K_h(x-y) \,  \Big(\frac{x-y}h\Big)^{\bf k+\bf k'} \, g(y)\, \dd y\, 
 \end{equation}
and 
\begin{equation} \label{eq:V} V = \E[\hat V] = \{V_{\bf k}\}_{{\bf k}\in {{\cal K}_l}}, \qquad \text{so}\quad V_{\bf k} = \int K_h(x-z) \, \Big(\frac{x-z}h\Big)^{\bf k}\,f(z)\, \dd z. \end{equation}
\begin{ass}[Boundedness]\label{ass:boundedness}
    For suitable $\rho_1>0$, and ${\cal U} = {\cal U}_{\rho_1}(x)$, we assume first that $$\sup_{y\in {\cal U}} \max\big\{f(y),g(y)\big\} \leq \rho_2$$ and second that there exist constants $\rho_3,\rho_4, \rho_5 \in (0,\infty)$ such that
\begin{equation} \label{eq:Boundary} 
\lambda_d\big(\big\{y\in \mathbb{R}^d \mid \|y-x\| < \varepsilon \, , \, g(y)\geq \rho_3\big\}\big) \, \geq \, \rho_4 \cdot \varepsilon^d \, , \quad \forall \varepsilon \in (0,\rho_5)\,. \end{equation}
%for constants .  
\end{ass}
This geometric condition includes boundary points $x$ with respect to the support of $g$, see Condition (X2) in \citet{viel2025convergence} for similar constraints. Note that \eqref{eq:Boundary} implies that the value $r(x)$ of the continuous version of $r$ is uniquely defined at $x$. 
%
%\hajo{Discuss special case of convex support of $g$ (locally at $x$)}
\begin{lemma}\label{lem:lowerboundeigen}
    Under the condition \eqref{eq:Boundary} from Assumption \ref{ass:boundedness} the matrix $S$ in \eqref{eq:Sdef} is positive definite, with smallest eigenvalue $\lambda_{\min}(S) \geq \underline{\lambda}>0$, with $\underline{\lambda}$ uniform in $h \in (0, \rho_5)$ and only depending on $\rho_3$, $\rho_4$, $d$ and $l$. 
\end{lemma}
This is restated as Lemma \ref{lem:Slowereigagain} in Section \ref{sec:somepreplemmas}, which is proved in Section \ref{sec:proofslemmas}. %We may assume $\underline{\lambda} \leq 1$ in the following. 
% {sec:proofremaingrates}
Since by Lemma \ref{lem:lowerboundeigen} the matrix $S$ is positive definite, we can set
\begin{equation}\label{eq:barrx}
 \bar r(x) := {\bf e}_0^\top S^{-1} V 
 \end{equation}
The overall estimation error $\hat r(x) - r(x)$ is then decomposed into the approximation error $\bar r(x) - r(x)$, which we bound in
Proposition \ref{prop:bias} below, as well as the stochastic error $\hat r(x) - \bar r(x)$.  
\end{remark}

\begin{remark}[Order $l=0$: ratio of kernel density estimators with equal bandwidths]%\label{rem:l0}
For $l=0$ we get $\mathcal K_0=\{0\}$, hence
$\hat S=\frac1m\sum_{j=1}^m K_h(x-Y_j)$ and $\hat V=\frac1n\sum_{i=1}^n K_h(x-X_i)$, so that
\[
  \hat r(x)\;=\;\frac{\frac1n\sum_{i=1}^{n}K_h(x-X_i)}
                     {\frac1m\sum_{j=1}^{m}K_h(x-Y_j)}.
\]
Thus the estimator $\hat r$ reduces to the ratio of two kernel density estimators, both with 
kernel $K$ and with the same bandwidth $h$, as suggested in  \citet{KelsallDiggle1995a,KelsallDiggle1995b}, see also \citet{kpotufe2017lipschitz}. Note that we do not need to assume that $K$ is a normalized probability density. %We  Note that the normalising constant of $K$
%cancels, so that $K$ need not integrate to one.
%
\end{remark}

\begin{remark}[Local least-squares]\label{rem:localLS}
The estimator in  \eqref{eq:esmatvec} can be derived from a local least squares criterion which makes its analogy to local polynomial regression particularly transparent. %and makes it transparent why smoothness of $r$ alone --
%and none of $f$ and $g$ -- is required. 
Let
\[
  Q_\theta(z)\,=\, \theta^\top \, Q(z) \,=\, \sum_{{\bf k}\in {\cal K}_l}\theta_{\bf k}\,z^{\bf k},
  \qquad z\in\mathbb R^d,
\]
for a coefficient vector $\theta=(\theta_{\bf k})_{{\bf k}\in {\cal K}_l}$ and $Q(z)=(z^{\bf k})_{{\bf k}\in {\cal K}_l}$, and consider 
\begin{equation}\label{eq:M}
  {\cal M}(\theta)\,=\,\int \,
     \Big(\,r(z)-Q_\theta\Big(\frac{x-z}{h}\Big)\Big)^{2}\, K_h(x-z)\, g(z)\,\dd z .
\end{equation}
To minimize ${\cal M}(\theta)$, expanding the square and using $f = r\, g$ we obtain
\begin{equation}\label{eq:Mexp}
  {\cal M}(\theta)\,=\,\theta^\top S\,\theta - 2\,\theta^\top V + \int K_h(x-y)r^2(y)g(y)\,\dd y,
\end{equation}
with $S$ as  defined in \eqref{eq:Sdef} and $V$ in \eqref{eq:V}. Since by Lemma \ref{lem:lowerboundeigen}, $S$ is positive definite under assumption \eqref{eq:Boundary},  ${\cal M}(\theta)$ has the unique minimizer $\bar \theta = S^{-1} \, V$, resulting in the approximation 
$ Q_{\bar \theta}(0) = {\bf e}_0^\top S^{-1} \, V=\bar r(x)$
of $r(x)$.
%The step from \eqref{eq:Qh} to \eqref{eq:Qhquad} uses the identity $r\,g=f$,
%which removes the unknown $r$ from the $\theta$-dependent part of the criterion:
%the quadratic term is an expectation with respect to $g$, the linear term one
%with respect to $f$, and both are estimated by plain sample averages.
%

To connect to local polynomial regression, for a response - predictor pair $(Y,X)$  with density $f_{Y,X}(y,z)$ w.r.t.~$\nu \otimes \lambda^d$, let $f_X$ denote the marginal density of $X$, and $m(z)=E[Y\mid X=z]$ the regression function. The analogous criterion to \eqref{eq:M} then is
\begin{align}
  {\cal M}_r(\theta)& \,=\,\int \int \,
     \Bigl(\,y-Q_\theta\Big(\frac{x-z}{h}\Big)\Bigr)^{2}\, K_h(x-z)\, f_{Y,X}(x,z)\,\dd \nu(y)\, \dd z \nonumber\\
     & \,=\,\int \,
      Q_\theta^2\,\Big(\frac{x-z}{h}\Big)\, K_h(x-z)\, f_{X}(z)\, \dd z \, -2 \, \int \,
      m(z)\, Q_\theta\,\Big(\frac{x-z}{h}\Big)\, K_h(x-z)\, f_{X}(z)\, \dd z \, + \, \text{const},\label{eq:Mr}
\end{align}
where the linear term arises from $\int y\,f_{Y,X}(y,z)\,\dd \nu(y)=m(z)\, f_X(z)$, which is analogous to $f=r\,g$, and the final constant term does not depend in $\theta$.  In local polynomial
regression, however, sample versions of both linear and quadratic terms in \eqref{eq:Mr} involve the same data (the covariates), while in \eqref{eq:Mexp},  $\hat S$ uses $Y_1,\dots,Y_m$ and
$\hat V$ uses $X_1,\dots,X_n$. 

Now the population quantities in \eqref{eq:M} which involve $\theta$ can be estimated by empirical counterparts, which after dropping the third term - which does not depend on $\theta$ - yields
\begin{equation}\label{eq:Memp}
  \widehat {\cal M}(\theta)
  \;=\;\frac1m\sum_{j=1}^{m} K_h(x-Y_j)\,
        Q_\theta^2\Bigl(\frac{x-Y_j}{h}\Bigr)
     \;-\;\frac2n\sum_{i=1}^{n} K_h(x-X_i)\,
        Q_\theta\Bigl(\frac{x-X_i}{h}\Bigr)
  \;=\;\theta^\top\hat S\,\theta-2\,\theta^\top\hat V .
\end{equation}
If $\hat S$ is non-singular and hence positive definite, the unique minimizer of 
$\widehat {\cal M}$ is $\hat\theta=\hat S^{-1}\hat V$, yielding the estimator 
$Q_{\hat\theta}(0)={\bf e}_0^\top\hat S^{-1}\hat V = \hat r(x)$ in \eqref{eq:esmatvec}. 

Minimizing least-squares criteria such as \eqref{eq:Memp} is known in the machine learning literature as Least-Squares Importance Fitting (LSIF), see e.g.~\citet{KanamoriHidoSugiyama2009}, the monograph
\citet{SugiyamaSuzukiKanamori2012}, and \citet{XYH25} for a recent contribution. However, the kernel-localized version that we focus on apparently to the best of our knowledge has not been previously studied.

\end{remark}

%\begin{rem}[Analogy to local polynomial regression estimation]
%
%Let us describe the analogy of density ratio and regression local polynomial estimation. Here suppose that $(Y,X), (Y_1, X_1), \ldots, (Y_n, X_n)$ are i.i.d.~with $Y$ being the real-valued response with finite expected value, $X$ the $d$-dimensional covariates with Lebesgue density $f_X$. Let $m(x) = \E[Y \mid X=x]$ denote the regression function. Suppose that $(Y,X)$ has density $f_{Y,X}$ w.r.t.~the product measure $\nu \otimes \lambda^d$, where $\nu$ is a $\sigma$-finite measure. Then $m(x) = \int_\R y\, f_{Y,X}(y,x)\, \dd \nu(y) / f_X(x)$ is a quotient of two functions, and the local polynomial regression estimator of order $l$ is of the form $\hat m(x) = {\bf e}_0^\top\, \hat{S}^{-1}\, \hat V$, with the Gram matrix $\hat{S}$ as in \eqref{eq:gramemp} but using  covariates $X$, and $\hat V_{\bf k}(x;h) = \frac1n \sum_{j=1}^n Y_j\, {K}_h(x-X_j)\, \Big(\frac{x-X_{j}}h\Big)^{\bf k}$, $ {\bf k} \in {\cal K}_l$.
%
%\end{rem}
%

\begin{remark}[Derivative estimation]
In contrast to LSIF and nearest neighbor methods, the local polynomial estimation approach also allows for direct estimation of the partial derivatives of the density ratio $r$: for ${\bf k} \in {\cal K}_l$ we let 
\begin{equation}\label{eq:derest}
\hat r_{{\bf k}}(x)  = \frac{{\bf k}!}{(-h)^{|{\bf k}|}}\, {\bf e}_{{\bf k}}^\top\, \hat{S}^{-1}\, \hat V
\end{equation}
with $\hat S$ and $\hat V$ defined in \eqref{eq:gramemp} and \eqref{eq:estdensrat}.
\end{remark}

\section{Rates of convergence} \label{sec:rates}
%%%%%%%%%%%%%%%%%%%%%%%%%%%%%%%%%%%%%%%%%%%%%%%%%%%%%%%%%%%%%%%%%%%%%

Here we establish the convergence rates for the estimation of the density ratio $r$ over Hölder smoothness classes. The proofs of Proposition \ref{prop:bias}, Theorem \ref{T:1} and Corollary \ref{C:1} below are provided in Section \ref{sec:proofsthrates}, while further proofs for results in this section are deferred to Section \ref{sec:proofremaingrates}.  

Let's start by introducing the Hölder classes that we consider. Assume that for a fixed $\beta>0$, $r$ is $(\lceil\beta\rceil-1)$-fold continuously differentiable on ${\cal U}$, where ${\cal U} = {\cal U}_{\rho_1}(x)$ is as in Assumption \ref{ass:boundedness} which is  assumed to hold in the following. Moreover, 
assume that $r$ and its partial derivatives 
$$r^{({\bf k})} := \partial^{|{\bf k}|}r / (\partial z_1^{k_1}\cdots \partial z_d^{k_d}), \qquad  
{\bf k}  \in {\cal K}_{\lceil\beta\rceil-1},$$ 
are bounded on ${\cal U}$ by some fixed constant $\rho_6$ and that
\begin{equation} \label{eq:Hoelder} \big|r^{({\bf k})}(y) - r^{({\bf k})}(x)\big| \, \leq \, \rho_6 \cdot \|x-y\|^{\beta - \lceil\beta\rceil + 1}\, , \quad \forall \, y\in {\cal U}\,, {\bf k} \in \mathbb{N}_0^d\mbox{ with }\sum_{j=1}^d k_j = \lceil\beta\rceil -1\,. \end{equation} 
%
%Thus at points on the boundary of $g$ we require  
%\hajo{Briefly mention boundary points, restrictions on support of $g$ to allow $\beta$-Hölder extensions of $r$  }

The tuples of densities $(f,g)$ which satisfy both the boundedness Assumption \ref{ass:boundedness} as well as these Hölder smoothness assumptions by definition are the elements of the class 
$${\cal F}(x):= {\cal F}(x,\beta,d,\rho_1,\rho_2,\rho_3,\rho_4,\rho_5,\rho_6).$$ 
Then for the approximation bias of $\bar r(x)$ in \eqref{eq:barrx} we have the following.  

\begin{proposition}[Approximation bias]\label{prop:bias}
Suppose that the order satisfies $l \ge \lceil\beta\rceil-1$. Then there is a
constant $C>0$ depending only on $d$ and $l$ such that for $(f,g) \in {\cal F}(x)$, 
\begin{equation}\label{eq:boundbiasfinal}
 \big|\bar r(x) - r(x)\big|  \leq C\, \rho_2\, \rho_6 \, \underline{\lambda}^{-1} \cdot h^\beta,  \qquad h\in\big(0,\min(\rho_1,\rho_5)\big).
\end{equation}
\end{proposition}
In \eqref{eq:boundbiasfinal}, the restriction $h<\rho_5$ is needed for Lemma~\ref{lem:lowerboundeigen} and $h<\rho_1$ for the Taylor expansion on $\mathcal U$ from Assumption \ref{ass:boundedness}.  

Next we consider the mean squared error at a fixed point $x\in \mathbb{R}^d$ under the smoothness constraints $(f,g)\in {\cal F}(x)$ and the mean integrated squared error under a global smoothness condition. 
To bound the expected value we require a trimmed version of the estimator $\hat r(x)$,
\begin{equation}\label{eq:esttrimmed}
 \hat r^*(x) = \hat r(x) \, 1\big( \lambda_{\min}(\hat S) > a_m\big) + 1\big( \lambda_{\min}(\hat S) \leq a_m\big)
 \end{equation}
for a deterministic sequence $a_m \downarrow 0$. %\hajo{highlight necessity of trimming for MSE}

\begin{theorem}[Bound on the mean squared error] \label{T:1}
Suppose that the order $l \geq \lceil\beta\rceil-1$, that $a_m \leq \underline{\lambda}/2$, where $\underline{\lambda}$ is from Lemma \ref{lem:lowerboundeigen}, and that $h \in (0, \min(\rho_1, \rho_5))$.  Then,
\begin{equation}\label{eq:finitesamplebound}
    \sup_{(f,g) \in{\cal F}(x)} \, \E_{f,g}\big[ \big|\hat r^*(x) - r(x)\big|^2\big] \, \leq \, C\,\Big(   \frac1{n h^d} + \frac1{m h^d} +  a_m^{-2} \, \exp(-c\, m\, h^d) \big( 1+ \frac1{n h^d}\big)\, + h^{2 \beta}\,\Big),
\end{equation}
where the constants $C = C(\rho_2,\rho_3,\rho_4, \rho_6, l, d, \beta, K)$ and $c$ depend only on $\rho_2,\rho_3,\rho_4, \rho_6$, $l$, $\beta$, $K$ and $d$, but not on $x, n,m, h$ nor on $(f,g) \in \mathcal{F}$. 
\end{theorem}
\begin{corollary}[Rates of convergence] \label{C:1}
Suppose that $l \geq \lceil\beta\rceil-1$, that $h \asymp \min\{m,n\}^{-1/(2\beta+d)}$, and that the trimming sequence $a_m \downarrow 0$ satisfies $\log(1/a_m) = o\big(m^{2 \beta/ ( 2 \beta + d)} \big)$, e.g.~$a_m = m^{-\gamma}$ for any $\gamma >0$ will do. Then as $\min\{m,n\} \to \infty$, 
$$ \sup_{(f,g) \in{\cal F}(x)} \, \mathbb{E}_{f,g} \big[\big|\hat{r}^*(x) - r(x)\big|^2\big] \, = \, {\cal O}\big(\min\{m,n\}^{-2\beta/(2\beta+d)}\big)\,.$$
Furthermore, let ${\cal G}$ be a compact subset of $\mathbb{R}^d$. Then,
$$ \sup_{(f,g) \in{\cal F}({\cal G})} \, \mathbb{E}_{f,g} \Big[\int_{\cG} \big|\hat{r}^*(x) - r(x)\big|^2 \, \dd x\Big] \, = \, {\cal O}\big(\min\{m,n\}^{-2\beta/(2\beta+d)}\big)\,,$$
where ${\cal F}({\cal G}) := {\cal F}({\cal G},\beta,d,\rho_1,\rho_2,\rho_3,\rho_4,\rho_5,\rho_6)$ is given by 
\begin{equation}\label{eq:funcclassG}
  {\cal F}({\cal G}) \, := \, \Big\{(f,g) \mid (f,g) \in \bigcap_{x\in {\cal G}}{\cal F}(x) \Big\}\,.
% {\cal F}({\cal G}) \, := \, \Big\{(f,g) \mid (f,g) \in \bigcap_{x\in {\cal G}}{\cal F}(x)  \mbox{ and }  g(y) = 0\, , \, \forall y\in \mathbb{R}^d\backslash{\cal G}\Big\}\,.
 \end{equation}
\end{corollary}
For $x\in{\cal G}$, from \eqref{eq:Boundary} it follows that $\Pr(Y_1 \in B_\varepsilon(x)) \geq \rho_3\, \rho_4 \, \varepsilon^d$, $\varepsilon \in (0,\rho_5)$, hence ${\cal G}$ must be contained in the support of $\Pr_{Y_1}$.
%The condition $g=0$ in ${\cal F}({\cal G})$ ensures that the support of $g$ is contained in ${\cal G}$.  %\hajo{muss dann ${\cal G}$ nicht automatisch der support von $g$ sein?}
%\hajo{comment on g=0 outside of ${\cal G}$}
%
%\begin{remark}[Optimality]
 %   In the case of $m\asymp n$ we obtain the classical convergence rates in nonparametric curve estimation, which are known to be minimax-optimal in standard density estimation. This optimality easily extends to the current setting by considering an inner point $x$ and assuming $g$ to be known and constant (equal to $1$) on ${\cal U}$ so that $r=f$. Note that the data $X_1,\ldots,X_n$ form a sufficient statistic for $r$ under these constraints. 
    
 %   On the other hand, note that, under these constraints, the densities $f$ and $g$ themselves cannot be estimated at $x$ with any convergence rate uniformly over $(f,g)\in {\cal F}(x)$. This can be seen as follows: just put $r\equiv 1$ (thus, perfectly smooth) so that $f=g$ and $X_1,\ldots,X_n,Y_1,\ldots,Y_m$ forms an i.i.d.~sample with the density $f=g$, on which no smoothness conditions are imposed. 
%In contrast to many other approaches to density ratio estimation such as \citet{kpotufe2017lipschitz} or \citet{XYH25} for unbounded density ratios, our procedure attains the optimal convergence rates without any logarithmic deterioration. 
%
%\end{remark}
%
%
\begin{remark}[Optimality]\label{rem:optimality}
The rate of Corollary~\ref{C:1} for $r(x)$ is minimax-optimal over ${\cal F}(x)$, for
every combination of $m$ and $n$. 
Consider an inner point $x$. 

Let, first, $g$ be fixed, constant equal to $1$, so that $r=f$. Hence the classical result for
density estimation over H\"older classes gives the lower bound of order
$n^{-2\beta/(2\beta+d)}$. Note that the data $X_1,\ldots,X_n$ form a sufficient statistic for $r$ under these constraints, since the law of $Y_1,\ldots,Y_m$ does not depend on $r$. 

Second, let $f$ be fixed, constant equal to $1$ so that  $g = 1/r$. If the range of $r$  is contained in the compact interval $[1/\rho_2, 1/\rho_3]$, where we may take $\rho_3 \leq 1/2$ and $\rho_2 > 2$, then $g$ has values in $[\rho_3, \rho_2]$. The Hölder class for $r$ at $x$ results in a corresponding Hölder class for $g$ at $x$, and the mean squared errors for $r(x)$ and for $g(x)$ of estimators contained in the appropriate range are of the same order. Hence the known lower bound of order $m^{-2\beta/(2\beta+d)}$ for estimating $g(x)$ transfers to $r(x)$, giving the overall lower bound of order  $\max\big(n^{-2\beta/(2\beta+d)},m^{-2\beta/(2\beta+d)}\big) = \min(m,n)^{-2\beta/(2\beta+d)}$ for the minimax risk. 

In contrast to many other approaches to
density ratio estimation such as \citet{kpotufe2017lipschitz} or
\citet{XYH25}, our procedure attains this rate
without any logarithmic deterioration.

\end{remark}

\begin{remark}[Non-estimability of individual densities over $\cF(x)$]\label{rem:nonestimable}
Suppose that $\rho_6 \geq 1$, so that pairs $(g,g)$ which have $r=1$ can be contained in $ \cF(x)$. Consider the subclass
$$\cF_c(x) \, = \, \big\{(f,g) \in \cF(x) \mid f \mbox{ and } g \mbox{ are continuous on } B_{\rho_1}(x)\big\}\,$$
of ${\cal F}(x)$, on which the values $f(x)$ and $g(x)$ are uniquely defined. Further suppose that the parameters $\rho_j$ are such that there exists a $g_0$ with $(g_0,g_0) \in {\cal F}_c(x)$ for which $\rho_3 < g_0 < \rho_2$ on $B_{\rho_1}(x)$.

In this case we show that $g(x)$ cannot be estimated consistently over this subclass: There exists $a>0$ such that 
for all $n, m$, we have that
\begin{equation}\label{eq:nonestimable}
\inf_{\hat g} \, \sup_{(f,g) \in {\cal F}_c(x)} \, \E_{f,g}\big[ |\hat g(x) - g(x)|^2\big] \, \geq \,a, 
\end{equation}
% a = 4\, \sqrt a
where the infimum is over all estimators $\hat g$ based on the samples $X_1, \ldots, X_n$ and $ Y_1, \ldots, Y_m$. A similar statement holds true for $f(x)$.

To show \eqref{eq:nonestimable}, by continuity there exist $b >0$ such that $\rho_3 + b < g_0 < \rho_2 - b$ on $B_{\rho_1}(x)$. Choose a function $\phi: \R^d \to [0,1]$, continuous with support in $B_1(0)$ and satisfying $\phi(0) = 1$. Fix $v\in\R^d$ with $\|v\|_2 = 1$, and for $\delta \in (0, \rho_1/3)$ set 
$$ g_{1,\delta}(y) \, = \, g_0(y) + b\, \phi\big((y-x)/\delta\big) - b\, \phi\big((y-x-2\delta \, v)/\delta\big)\,. $$
Then $g_{1,\delta}$ is a continuous probability density which coincides with $g_0$ outside $B_{3\, \delta}(x)$ and for which also  $(g_{1,\delta}, g_{1,\delta}) \in \cF(x)$, as well as $g_{1,\delta}(x) - g_0(x) = b$ for all $\delta\in (0, \rho_1/3)$. 
Further the $\chi^2$-divergence is upper bounded by
$$\chi^2(g_{1,\delta} \,\|\, g_0) \, = \, \int \frac{(g_{1,\delta}-g_0)^2}{g_0} \, \leq \, \frac{2\,b^2\, \delta^d}{\rho_3} \int \phi^2 \,.$$
 Under $(g,g)$ the pooled sample $X_1, \ldots, X_n,  Y_1, \ldots, Y_m$ is i.i.d.~with density $g$. Choosing $\delta$ so small that $\chi^2(g_{1, \delta}\|g_0)) \leq 1/(n+m)$ then concludes the proof of \eqref{eq:nonestimable} by applying standard decision theory \citep[Chapter~2]{tsybakov2009introduction}.

\end{remark}

\begin{remark}[Rate for derivative estimation]
For the estimator $\hat r_{{\bf k}}(x)$ of the ${\bf k}$-th partial derivative in \eqref{eq:derest}, for $|{\bf k}| \geq 1$ we use the 
trimming 
\begin{equation}\label{eq:esttrimmedder}
 \hat r_{{\bf k}}^*(x) = \hat r_{{\bf k}}(x) \, 1\big( \lambda_{\min}(\hat S) > a_m\big).  
 \end{equation}
 Then is $|{\bf k}| \leq \lceil\beta\rceil-1$,  $l \geq \lceil\beta\rceil-1$, $h \asymp \min\{m,n\}^{-1/(2\beta+d)}$, and the trimming sequence $a_m \downarrow 0$ is chosen as in Corollary \ref{C:1}, as $\min\{m,n\} \to \infty$ we have that 
 \begin{equation}\label{eq:dervest}
 \sup_{(f,g) \in{\cal F}(x)} \, \mathbb{E}_{f,g} \big[\big| \hat r_{{\bf k}}^*(x) - r^{({\bf k})}(x)\big|^2\big] \, = \, {\cal O}\big(\min\{m,n\}^{-2\, (\beta - |k|)/(2\beta+d)}\big)\,.
 \end{equation}
 The proof is outlined in Section \ref{sec:proofremaingrates}. 
\end{remark}
Next we turn to a concentration inequality for the estimator, which is useful for applications in classification \citep{audibert2007fast}. 
\begin{theorem}[Concentration inequality]\label{th:concinequ}
    Under Assumption \ref{ass:boundedness} for $h \in (0, \min(\rho_1,\rho_5))$ we have for $\delta>0$ that 
    \begin{align}
     \Pr\Big(\big|\hat r(x) - \bar r(x)\big| \geq \delta \Big)  \, \leq \, & 2 \# {\cal K}_l\, \exp\big(-c_1\, n\, h^d\, \min(\delta,\delta^2)\big)\nonumber \\
      + & 2 \,(\# {\cal K}_l)^2\, \exp\big(-c_2\, m\, h^d\, \min(\delta,\delta^2)\big) + \# {\cal K}_l\, \exp(-c_3\, m\, h^d),\label{eq:expbound}
    \end{align} 
    where $\bar r(x)$ is given in \eqref{eq:barrx}, and the constants $c_i$ depend only on $\rho_2,\rho_3,\rho_4$, $d$, $l$ and $K$.  
\end{theorem}
%
%The proof is also provided in Section \ref{sec:proofsthrates}. 
%
The same result holds for the trimmed estimator $\hat r^*(x)$ in \eqref{eq:esttrimmed} in case $a_m \leq \underline{\lambda}$.
 
%\begin{remark} \label{Rem:concen}
% Also holds for trimmed estimator, but here trimming not necessary.    
%\end{remark}
%
\begin{theorem}[Rate in the uniform norm]\label{T:unifrate}
    Let ${\cal G}$ be a compact subset of $\mathbb{R}^d$. Suppose that $l \geq \lceil\beta\rceil-1$, that additionally to compact support on $B_1(0)$ and boundedness between $0$ and $1$ the kernel $K$ is Lipschitz continuous,  
    that $h \asymp \big(\log(\min\{m,n\})/\min\{m,n\}\big)^{1/(2\beta+d)}$, and that the trimming sequence $a_m \downarrow 0$ satisfies $\log(1/a_m) = o\big(m^{2 \beta/ ( 2 \beta + d)} \big)$. Then as $\min\{m,n\} \to \infty$, 
$$ \sup_{(f,g) \in{\cal F}({\cal G})} \, \mathbb{E}_{f,g} \big[\sup_{x \in {\cal G}}\big|\hat{r}^*(x) - r(x)\big|^2\big] \, = \, {\cal O}\Big( \,\Big(\frac{\log(\min\{m,n\})}{\min\{m,n\}}\Big)^{2\beta/(2\beta+d)}\,\Big)\,,$$
with ${\cal F}\big({\cal G}\big)$ defined in \eqref{eq:funcclassG}. 
\end{theorem}
The uniform rate includes boundary points of the support of $g$. The proofs of these two theorems are deferred to Section \ref{sec:proofremaingrates}.
%\hajo{discuss: uniformity covers the boundary of the support of $g$}

\section{Asymptotic normality}\label{sec:asympnorm}
Next we turn to the asymptotic distribution of the estimator. We obtain asymptotic normality at individual points $x$ including boundary points of the support of $g$. Further, we propose an appropriate studentization, and under additional regularity assumptions on the individual densities $f$ and $g$ at $x$ we obtain explicit expressions for the asymptotic variance. Finally, asymptotic normality of the derivative estimators in \eqref{eq:derest} is also obtained. 

In addition to Assumption \ref{ass:boundedness} and continuity of $r$ on ${\cal U}$ we assume that 
\begin{equation}\label{eq:rpositive}
r(z) >0, \qquad z \in {\cal U}. 
\end{equation}
Together with $f = r\, g$ $\lambda^d$-a.e.~on ${\cal U}$ this implies that condition \eqref{eq:Boundary} also holds for $f$, 
\begin{equation} \label{eq:Boundary1} 
\lambda_d\big(\big\{z\in \mathbb{R}^d : \|z-x\| < \varepsilon \, , \, f(z)\geq \rho_3\big\}\big) \, \geq \, \rho_4 \cdot \varepsilon^d \, , \quad \forall \varepsilon \in (0,\rho_5)\,, \end{equation}
where we assume that the constants $\rho_3,\rho_4, \rho_5 \in (0,\infty)$ have been chosen to suit both $g$ and $f$.   

Set 
$$Q(z)=(z^{\bf k})_{{\bf k}\in {\cal K}_l}$$ and let
\begin{align}\label{eq:defpopquantasympnormcov}
\begin{split}    
 M & \, = \, M(x;h) \, = \, \int f(x-hz) \, K^2(z) \, Q(z)\,Q(z)^\top \dd z\,,\\
  N & \, = \, N(x;h) \, = \, \int g(x-h\,y) \, K^2(y) \, |\bar \theta^\top \, Q(y)|^2\,Q(y)\, Q(y)^\top \dd y\,,   
  \end{split}
\end{align}
with $\bar \theta = S^{-1}\, V$, see Remark \ref{rem:localLS}, as well as 
\begin{equation}\label{eq:covmatrun}
        \Sigma_{m,n} \,=\, \frac{\min(m,n)}n \,(M - h^d\, V\, V^\top) \, + \, \frac{\min(m,n)}m \, (N - \, h^d\, V\, V^\top). 
\end{equation}
and
        \begin{align}\label{eq:varaysmp}
        \sigma_{m,n}^2 & \,=\, {\bf e}_0^\top \, S^{-1}\, \Sigma_{m,n} \, S^{-1}\, {\bf e}_0\\
        & \,=\, \frac{\min(m,n)}n \,\big({\bf e}_0^\top \, S^{-1}\, M \, S^{-1}\, {\bf e}_0- h^d\, \bar r(x)^2\big) \, + \, \frac{\min(m,n)}m \,\big({\bf e}_0^\top \, S^{-1}\, N \, S^{-1}\, {\bf e}_0- h^d\, \bar r(x)^2\big).\nonumber
        \end{align}
Further introduce
\begin{align}\label{eq:quantvarest}
       \begin{split}           
           A_i & \, = \, K_h(x-X_i)\, {\bf e}_0^\top \,\hat S^{-1} \,  Q\Big(\frac{x-X_i}{h}\Big)\, ,\\
           B_j & \, = \, K_h(x-Y_j)\, {\bf e}_0^\top \, \hat S^{-1} \,  Q\Big(\frac{x-Y_j}{h}\Big) \, \hat \theta^\top\,  Q\Big(\frac{x-Y_j}{h}\Big)\, 
       \end{split}    
       \end{align}
and 
\begin{equation}\label{eq:varest}
    \hat \sigma_{m,n}^2 = \min(m,n)\, h^d\,\Big(\frac1{n^2}\, \sum_{i=1}^n \big(A_i - \hat r(x) \big)^2 \, + \, \frac1{m^2}\, \sum_{j=1}^m \big(B_j - \hat r(x) \big)^2 \Big).
\end{equation}
The proof of the following theorem is provided in Section \ref{sec:asympnormproofs}.
\begin{theorem}\label{th:CLTgen}
    Impose Assumption \ref{ass:boundedness} and additionally assume  \eqref{eq:rpositive}. Then as $h \to 0$ and $\min(m,n)\, h^d \to \infty$ we have that
    \begin{equation}\label{eq:CLTstand}
        \big(\min(m,n)\, h^d \big)^{1/2} \, \frac{\hat r(x) - \bar r(x)}{\sigma_{m,n}}\, \stackrel{{\cal L}}{\longrightarrow}\, {\cal N}(0, 1),
    \end{equation}
    with $\sigma_{m,n}^2$ defined in \eqref{eq:varaysmp}, which is bounded away from $0$. Here, $\stackrel{{\cal L}}{\longrightarrow}\, {\cal N}(0, 1)$ denotes weak convergence to a standard normal distribution. 
      
    Further,  \eqref{eq:CLTstand} remains true if we replace $\sigma_{m,n}$ by $\hat \sigma_{m,n}$, with $\hat \sigma_{m,n}^2$ defined in \eqref{eq:varest}.
\end{theorem}
\begin{remark}[Limits of the asymptotic variance]\label{rem:asympvarlimitform}
    If the individual densities $f$ and $g$ are positive and continuous in a neighborhood of $x$, and if  
    \begin{equation} \label{eq:mton}
m/n \, \to \, \kappa\,, \qquad \mbox{ as }n\to\infty\,
\end{equation} 
for some $\kappa \in [0,\infty]$, then
\begin{equation}\label{eq:asympvarform}
     \sigma_{m,n}^2 \longrightarrow \frac{r(x)}{g(x)} \, {\bf e}_0^\top\, S_K^{-1}\, \big( \min\{\kappa,1\} M_K  +  \min\{1,1/\kappa\} \cdot r(x) \, N_K\big)\, S_K^{-1}\,{\bf e}_0\,, 
\end{equation}
where $1/\infty =: 0$, and where
    \begin{align} 
M_K  & \, = \, \int K^2(z) \, Q(z)\, Q(z)^\top  \dd z\,, \quad S_K  \, = \, \int K(z) \, Q(z)\, Q(z)^\top  \dd z\,, \quad V_K \, = \, \int K(z) \, Q(z)\,   \dd z\,,\label{eq:matrices1} \\ 
N_K  & \, = \, \int K^2(z) \,|V_K^\top \, S_K^{-1}\, Q(z)|^2\,  Q(z)\, Q(z)^\top \dd z,\nonumber
\end{align}
and $\hat \sigma_{m,n}^2$ converges in probability to the right side of \eqref{eq:asympvarform}. 
Note that continuity of $g$ at $x$ with $g(x)>0$ requires $x$ to lie in the interior of the support of $g$. For boundary points of $g$ the convergence in \eqref{eq:asympvarform} can be generalized if the boundary at $x$ is sufficiently regular, e.g.~locally approximated by a cone $C$. In this case the integrals in \eqref{eq:matrices1} have to be taken over $C$ only. 

Using asymptotic formulas such as \eqref{eq:asympvarform} directly requires estimates and hence regularity of $f$ and $g$ at $x$. The approach in  Theorem \ref{th:CLTgen} avoids this and works under minimal assumptions on $f$ and $g$. 
\end{remark}

\begin{remark}[Asymptotic normality of derivative estimates]\label{rem:derivest}
    %
 %   $$D_{\cal I}:=\operatorname{diag}\Big(\frac{{\bf k}!}{(-h)^{|{\bf k}|}}\Big){{\bf k}\in{\cal I}},\qquad
%\hat r{\cal I}(x):=\big(\hat r_{\bf k}(x)\big){{\bf k}\in{\cal I}}=D{\cal I}E_{\cal I}^\top\hat\theta,\qquad
%\bar r_{\cal I}(x):=D_{\cal I}E_{\cal I}^\top\bar\theta,$$

%$$\Gamma_{\cal I}:=E_{\cal I}^\top S^{-1}\Sigma_{m,n}S^{-1}E_{\cal I}\in\R^{q\times q},\qquad
%\hat\Gamma_{\cal I}:=E_{\cal I}^\top\hat S^{-1}\hat\Sigma_{m,n}\hat S^{-1}E_{\cal I},$$
%
Theorem \ref{th:CLTgen} extends to the estimates of the derivatives in \eqref{eq:derest}. 
Fix ${\bf k} \in {\cal K}_l$, and define $\sigma_{m,n;{\bf k}}^2$ analogously as $\sigma_{m,n}^2$ in \eqref{eq:varaysmp} by replacing ${\bf e}_0$ by ${\bf e}_{{\bf k}}$. Similarly, in the definition of $A_i$ and $B_j$ in \eqref{eq:quantvarest}, also replace ${\bf e}_0$ by ${\bf e}_{{\bf k}}$. Substituting these into \eqref{eq:varest} and also replacing $\hat r(x)$ by ${\bf e}_{{\bf k}}^\top \hat S^{-1}\, \hat V$ defines the estimator $\hat \sigma_{m,n;{\bf k}}^2$. 

Then under Assumption \ref{ass:boundedness} and additionally imposing \eqref{eq:rpositive} we obtain
 as $h \to 0$ and $\min(m,n)\, h^d \to \infty$ that
\begin{equation}\label{eq:asympnormder}
     \big(\min(m,n)\, h^{d+2\, |{\bf k}|}\big)^{1/2}\, \cdot
      \frac{\hat r_{{\bf k}}(x) - \bar r_{{\bf k}}(x)}{{\bf k}!\, \cdot \, \sigma_{m,n; {\bf k}}}
      \, \stackrel{{\cal L}}{\longrightarrow}\, {\cal N}(0, 1)\,,
\end{equation}
where $\bar r_{{\bf k}}(x) = {\bf k}!\, {\bf e}_{{\bf k}}^\top\, {S}^{-1}\,  V / (-h)^{|{\bf k}|}$, and \eqref{eq:asympnormder} continues to hold if we replace $\sigma_{m,n;{\bf k}}$ by $\hat \sigma_{m,n;{\bf k}}$. 
In Theorem \ref{th:CLTmultvarder} in Section \ref{sec:asympnormproofs} we formulate and prove a multivariate version of \eqref{eq:asympnormder}, which shows joint asymptotic normality of finitely many derivative estimators. 
\end{remark}

\section{Estimation of the Kullback-Leibler divergence}\label{sec:KLest}
%%%%%%%%%%%%%%%%%%%%%%%%%%%%%%%%%%%%%%%%%%%%%%%%%%%%%%%%%%%%%%%%%%%%%%%%%%%%%%%%%%%%%%%%%

Besides estimation of the density ratio there is interest in the related problem of estimating the Kullback-Leibler divergence between the densities $g$ and $f$, which may be written as
$$ \mathbb{K}(g,f) \, := \, \int \Big(\log \frac{g(x)}{f(x)}\Big) \, g(x) \, \dd x \, = \, - \int \big(\log r(x) \big) \, g(x) \, \dd x\,, $$
see e.g.~\citet{nguyen2010estimating, BS23} and related results on entropy estimation, e.g.~\citet{L96, TM96, BSY19}. Estimation of the Kullback-Leibler divergence is included in the more general setting of estimating the non-linear functional
\begin{equation} \label{eq:KL} \tilde{\Phi} \, := \, \tilde{\Phi}(f,g) \, := \, \int \Phi\big(r(x)\big) \, g(x) \, \dd x\,, \end{equation}
when we set $\Phi(x) := - \log(x)$.

Now consider an approach that involves split of the sample and debiasing. Concretely we use the partial dataset ${\cal D}^{[1]}_{n^*,m^*} := \{X_j \mid  j=1,\ldots,n^*\} \cup \{Y_k  \mid k=1,\ldots,m^*\}$, for some integers $n^*<n$ and $m^*<m$, to construct the truncated estimator 
$$\hat{r}^*_{1,t}(x) := \max\big\{\rho_7,\min\{\rho_8,\hat{r}^*_{1}(x)\}\big\},$$ 
where $\hat{r}^*_{1}(x)$ is the initial estimator defined as in \eqref{eq:esttrimmed} but based on ${\cal D}^{[1]}_{n^*,m^*}$ instead of the full data sample. Here, suppose that $\rho_7>0$ and $\rho_8$ are known constants such that
\begin{equation} \label{eq:bound_r}
r(x) \in [\rho_7,\rho_8]\,, \qquad \forall x\in {\cal G}\,,
\end{equation}
where in this section we take ${\cal G}$ as the full support of $\Pr_{Y_1}$. 
As a motivation for the estimator, note that for sufficiently smooth $\Phi$ from a first order Taylor approximation
\begin{align} \nonumber
\tilde{\Phi} & \, \approx \, \int \Phi\big(\hat{r}^*_{1,t}(x)\big) \, g(x) \, \dd x \, + \, \int \Phi'\big(\hat{r}^*_{1,t}(x)\big) \cdot \big(r(x) - \hat{r}^*_{1,t}(x)\big) \, g(x) \, \dd x \\
\label{eq:KL.2} & \, = \, \int \big\{\Phi\big(\hat{r}^*_{1,t}(x)\big)\, - \, \Phi'\big(\hat{r}^*_{1,t}(x)\big)\cdot \hat{r}^*_{1,t}(x)\big\} \, g(x) \, \dd x \, + \, \int \Phi'\big(\hat{r}^*_{1,t}(x)\big) \, f(x) \, \dd x\,.
\end{align}
This inspires us to employ the remaining dataset ${\cal D}^{[2]}_{n^*,m^*} := \{X_j \mid  j=n^* + 1,\ldots,n\} \cup \{Y_k  \mid k=m^* + 1,\ldots,m\}$ to define the estimator of $\tilde{\Phi}$ in \eqref{eq:KL} by
\begin{equation} \label{eq:estimatorPhi}
\hat{\Phi} \, := \, \frac1{m-m^*} \sum_{k=m^*+1}^m \hat {\Psi}_1(Y_k) \, + \, \frac1{n-n^*} \sum_{j=n^*+1}^n \hat{\Psi}_2(X_j)\,,
\end{equation}
 where
\begin{align*}
\hat{\Psi}_{1}(x) & \, := \, \Phi\big(\hat{r}^*_{1,t}(x)\big)\, - \, \Phi'\big(\hat{r}^*_{1,t}(x)\big)\cdot \hat{r}^*_{1,t}(x)\,, \hspace{3cm} \hat{\Psi}_2(x) \, := \,  \Phi'\big(\hat{r}^*_{1,t}(x)\big)\,.
\end{align*}
Moreover we introduce the deterministic counterparts
\begin{align*}
\Psi_1(r;x) & \, := \, \Phi\big(r(x)\big)\, - \, \Phi'\big(r(x)\big)\cdot r(x)\,, \hspace{3cm}
\Psi_2(r;x) \, := \,  \Phi'\big(r(x)\big)\,.
\end{align*}

The following theorem provides an upper bound on the asymptotic risk of the estimator $\hat{\Phi}$. The proof is provided in Section \ref{sec:proofKL}. 
\begin{theorem} \label{T:KL.1}
Consider the estimator $\hat{\Phi}$ in (\ref{eq:estimatorPhi}) under the conditions of Corollary \ref{C:1} while the bandwidth selection is changed to $h \asymp (n^*)^{-1/(2\beta+d)}$. Impose $m/n \to \kappa\in (0,\infty)$ as $n \to \infty$. Grant that $\beta>d/2$. Assume that $\Phi$ is twice continuously differentiable on the interval $[\rho_7,\rho_8]$ where $|\Phi|$, $|\Phi'|$ and $|\Phi''|$ are uniformly bounded on this domain. Choose $m^*\asymp m/\log m$ and $n^*\asymp n/\log n$. Fix some compact subset ${\cal G}$ of $\mathbb{R}^d$. Then, for all $(f,g) \in {\cal F}({\cal G}) \cap \{g \mid \text{supp}\, (g\, \dd \lambda^d)= {\cal G}\}$ for which the density ratio $r$ satisfies \eqref{eq:bound_r}, it holds that
$$ \lim_{n\to\infty} \, n \cdot \mathbb{E}_{(f,g)} \big[\big|\hat{\Phi} \, - \, \tilde{\Phi}(f,g)\big|^2\, \big]\, = \,  {\cal V}(f,g) \, := \, \kappa^{-1}\cdot \mbox{var}_g \, \Psi_1(r;Y_1) \, + \, \mbox{var}_f\, \Psi_2(r;X_1)\,. $$
Moreover, whenever ${\cal V}(f,g) \neq 0$, the random sequence
$$ \big\{\sqrt{n}\cdot \big(\hat{\Phi} - \tilde{\Phi}(f,g)\big)\big\}_{n\in \mathbb{N}} $$
converges in distribution to a centered Gaussian random variable with the variance ${\cal V}(f,g)$. 
\end{theorem}

\begin{remark}
Theorem \ref{T:KL.1} shows that the parametric convergence rate $n^{-1}$ (with respect to the mean squared error) is attainable when $r$ is sufficiently smooth; precisely when $\beta>d/2$. 
This condition is also required in \citet{nguyen2010estimating}. 
Concerning $n^*, m^*$ it suffices to choose these such that  $n^*/n\to 0$, $m^*/m \to 0$ as well as $(n^*)^{-4\beta/(2\beta+d)} \, = \, o(n^{-1})$ and similarly for $m^*$ and $m$, the choice from the theorem works for all parameters $\beta>d/2$.   

For $\beta\leq d/2$ upon choosing $n^* \asymp n$ and $m^* \asymp m$ we instead obtain the rate ${\cal O}\big(\min\{n,m\}^{-4\beta/(2\beta+d)}\big)$. 
In \citet{BS23}, when $d\leq 3$, the smoothness constraint is weakened to $\beta>d/4$, in which case the parametric rate still applies. See also \citet{L96}. However the conditions are not directly comparable since, in \citet{BS23}, smoothness is imposed both on $f$ and on $g$ rather than only on the ratio $r$ as in the current setting. With respect to the Kullback-Leibler divergence (i.e.~$\Phi(x)=-\log(x)$) the asymptotic constants from Theorem \ref{T:KL.1} coincide with those derived in Example 1 in \citet{BS23} where asymptotic efficiency results can be found in Theorem 14 of \citet{BS23}. 
\end{remark}

\section{Proofs}

Here we present proofs of some of the main results in Section \ref{sec:rates}. Most details and further proofs are collected in the Appendix.

In the proofs we drop the order $l$ in the notation and write ${\cal K} = {\cal K}_l$. Moreover, we write $\E = \E_{f,g}$ for some fixed pair $(f,g) \in {\cal F}(x)$.

%%%%%%%%%%%%%%%%%%%%%%%%%%%%

%
%
%
%

\subsection{Preparatory lemmas}\label{sec:somepreplemmas}
We start with several preparatory lemmas. The proofs are provided in Section \ref{sec:proofslemmas}. 
\begin{lemma}\label{lem:Slowereigagain}
    Suppose that the kernel $K : \R^d \to [0,1]$ is supported on $B_1(0)$ and satisfies $K(z)>0$ for $\|z\|<1$ and that the density $g$ satisfies assumption \eqref{eq:Boundary} at $x$. For a degree parameter $l \in \N_0$ set  $Q(z)=(z^{\bf k})_{{\bf k}\in {\cal K}_l}$, and consider the $\# {\cal K}_l \times \# {\cal K}_l$ symmetric matrix
    $$ S \, =\, S(x;h) \, =\, \int\, K(y)\, Q(y)\, Q(y)^\top g(x - y\, h)\, \dd y.$$
    Then $\lambda_{\min}(S) \geq \underline{\lambda} >0$, with $\underline{\lambda}$ uniform in $h \in (0, \rho_5)$ and only depending on $\rho_3$, $\rho_4$, $d$ and $l$.
\end{lemma}
This is simply Lemma \ref{lem:lowerboundeigen} restated for clarity, since it is later also applied with kernel function $K^2$, density $f$ and with degree $2\cdot  l$.
\begin{lemma}\label{lem:tsybakov2}
Under Assumption \ref{ass:boundedness}, for $h \in (0, \rho_5)$ we have that
    $$\Pr\big(\lambda_{\min}(\hat S) \leq \underline{\lambda}/2\big) \leq (\# {\cal K})\, \exp(-c\, m\, h^d),$$
    where $\underline{\lambda}$ is as in Lemma \ref{lem:lowerboundeigen}, and the constant can be taken as $c = \big(\underline{\lambda}/(2\, (\# {\cal K}) \big)\, \log(\mathrm{e}/2) $, where $\mathrm{e}$ is the Euler number. 
\end{lemma}
\begin{lemma}\label{lem:boundsterms}
    Under Assumption \ref{ass:boundedness}, for $h \in (0, \rho_1)$ we have that
    \begin{align}
    \begin{split}\label{eq:stochbound1}
    \E\big[ \big\|\hat V - V\big\|_2^2\big] & \leq \# {\cal K}\, \rho_2\, \frac1{n h^d} , \qquad \|V\|_2^2  \leq \# {\cal K}\, \rho_2^2,, \qquad \E\big[\|\hat V\|_2^2\big]  \leq\, \# {\cal K}\, \rho_2\, \big(\rho_2 + \frac1{n h^d} \big),\\
    \E\big[\big\| S - \hat S \big\|_F^2\big] &\leq (\# {\cal K})^2\, \rho_2\, \frac1{m h^d}, \qquad  \| S \|_{\op}^2 \leq \| S \|_{F}^2 \leq (\# {\cal K})^2\, \rho_2^2.
    \end{split}
    \end{align}
\end{lemma}

\subsection{Proofs of Proposition \ref{prop:bias}, Theorem \ref{T:1} and Corollary \ref{C:1} from Section \ref{sec:rates}}\label{sec:proofsthrates}

\begin{proof}[Proof of Proposition \ref{prop:bias}]
%We discuss the approximation bias of $\bar r(x) = P_\theta(0) = {\bf e}_0^\top S^{-1} \, V$, the population version of the estimator $\hat r(x)$, see Remark \ref{rem:localLS}, and \eqref{eq:Sdef} and \eqref{eq:V} for the definitions of $S$ and $V$.  
%
Let
\begin{equation}\label{eq:paramvecpop}
 \bar \alpha = S^{-1} {\bf e}_0
 \end{equation}
with associated kernel
\begin{equation*} %\label{eq:Kbar} 
\bar{K}(z;x,h) \, := \, K(z)\, \sum_{{\bf k}\in {\cal K}} \bar{\alpha}_{\bf k} \cdot z^{\bf k}, \end{equation*}
where as mentioned above, ${\cal K} = {\cal K}_l$. 
In view of \eqref{eq:V} we can write
\begin{equation}\label{eq:convrep}
    \big[\bar{K}_h*f\big](x) = \bar \alpha^\top \,V = {\bf e}_0^\top S^{-1} \, V = \bar r(x).
\end{equation}

Set $s= \lceil\beta\rceil-1$. A Taylor expansion gives

\begin{align} \nonumber 
\big|\bar r(x) - r(x)\big| &\, = \,  \big|\big[\bar{K}_h*f\big](x) - r(x)\big|   \, = \,  \Big|\int \bar{K}(z) g(x-zh) r(x-zh) \, \dd z -r(x)\Big| \\ \nonumber 
& \, = \,  \Big|\sum_{{\bf k}\in {\cal K}_s} \frac{(-h)^{|{\bf k}|} }{{\bf k}!} \cdot r^{({\bf k})}(x) \int \bar{K}(z) \, z^{\bf k}\,  g(x-zh) \dd z \, - \, r(x) \, + \, {\cal R}\Big|\,\\
& \, = \,  \Big|\sum_{{\bf k}\in {\cal K}_s} \frac{(-h)^{|{\bf k}|} }{{\bf k}!} r^{({\bf k})}(x) \cdot T_{\bf k} \, - \, r(x) \, + \, {\cal R}\Big|\,\,,\label{eq:Taylor}
\end{align}
with the remainder term
\begin{align} \nonumber {\cal R}  \, := \,  & \sum_{|{\bf k}|=\lceil\beta\rceil-1} \frac{(-h)^{|{\bf k}|}}{{\bf k}!} \int \bar{K}(z)\, g(x-zh) \, z_1^{k_1}\cdots z_d^{k_d} \\ \label{eq:remainder0} & \cdot \int_0^1 (\lceil\beta\rceil-1)\cdot(1-u)^{\lceil\beta\rceil-2} \cdot \big(r^{({\bf k})}(x-u z h) - r^{({\bf k})}(x)\big) \, \dd u \, \dd z\,. \end{align}
for $s \geq 1$. 
Now in view of \eqref{eq:paramvecpop},
$$ T_{{\bf k}} \, := \,  \int \bar{K}_h(x-y) \, g(y) \Big(\frac{x-y}h \Big)^{\bf k} \, dy \, \, =\, {\bf e}_{\bf k}^\top\, S \bar \alpha =  1(\bf k=0)$$
for all ${\bf k} \in {\mathcal K}_l$, and since $s \leq l$ in particular for ${\bf k} \in {\mathcal K}_s \subseteq {\mathcal K}_l$
so that all terms in \eqref{eq:Taylor} except the remainder $\cal R$ vanish. 

For $s=0$ we simply have
\begin{align*} \nonumber \bar r(x) - r(x) \,=\, {\cal R}  \, = \,  &  \int \bar{K}(z)\, g(x-zh) \, \big(r(x-z h) - r(x)\big)  \, \dd z\,. \end{align*}

Now, in either case, for the remainder ${\cal R}$, 
\begin{align} \nonumber  |{\cal R}| & \, \leq \, \rho_2 \rho_6 \cdot h^\beta \cdot \int |\bar{K}(z)| \, \frac{\|z\|^{\beta-\lceil\beta\rceil+1}}{(\lceil\beta\rceil-1)!} \, \Big(\sum_{j=1}^d|z_j|\Big)^{\lceil\beta\rceil-1} \\ \nonumber & \hspace{3cm} \cdot\sum_{k_1+\cdots+k_d=\lceil\beta\rceil-1} {\lceil\beta\rceil - 1 \choose k_1,\ldots,k_d}  \prod_{j=1}^d \Big(\frac{|z_j|}{|z_1|+\cdots+|z_d|}\Big)^{k_j}  dz \\
\label{eq:remainder} & \, \leq \, \rho_2\rho_6 \cdot h^\beta \cdot \frac{(\sqrt{d})^{\lceil\beta\rceil-1}}{(\lceil\beta\rceil-1)!}  \, \int |\bar{K}(z)| \cdot \|z\|^\beta \dd z\,, \end{align}
for $h\in(0,\rho_1)$ where we use the multinomial distribution and that $|z_1|+\cdots|z_d|\leq \sqrt{d} \|z\|$.
From Lemma \ref{lem:lowerboundeigen} it follows for  the weights  $\bar \alpha$ in \eqref{eq:paramvecpop} that  $\|\bar \alpha\|\leq \underline{\lambda}^{-1}$. Since $K$ is supported on the closed Euclidean unit ball $B$ and is $\leq 1$, the kernel function $\bar K$ is uniformly upper-bounded by $(\#\, \mathcal K)^{1/2}\, \underline{\lambda}^{-1}$ for $h \in (0, \min(\rho_1,\rho_5))$. Hence 
$$\int |\bar{K}(z)| \cdot \|z\|^\beta \dd z \leq (\#\, \mathcal K)^{1/2}\, \underline{\lambda}^{-1} \lambda^d(B),$$
and we obtain   the claim \eqref{eq:boundbiasfinal} from \eqref{eq:Taylor} and \eqref{eq:remainder}.
\end{proof}
\begin{proof}[Proof of Theorem \ref{T:1}]
Assume that $m$ is so large that $a_m \leq \underline{\lambda}/2$. We may then decompose the error as 
\begin{align*}
    \hat r^*(x) - r(x)  = &\, 1\big(\lambda_{\min}(\hat S) > \underline{\lambda}/2\big)\, {\bf e}_0^\top \, \hat S^{-1} \, \big(\hat V - V\big) + 1\big(\lambda_{\min}(\hat S) > \underline{\lambda}/2\big)\, \, {\bf e}_0^\top \, \big( \hat S^{-1} - S^{-1}\big) \, V \\ 
    & \, + \, {\bf e}_0^\top \, \hat S^{-1} \, \hat V\,  1\big(a_m < \lambda_{\min}(\hat S) \leq \underline{\lambda}/2\big)\, + \, 1\big(\lambda_{\min}(\hat S) \leq a_m\big) \, - \, 1\big(\lambda_{\min}(\hat S) \leq  \underline{\lambda}/2\big) \,{\bf e}_0^\top \, S^{-1}\, V \\
    & \, + \, \bar r(x) - r(x),
\end{align*}
where we used \eqref{eq:convrep}.
From $\hat S^{-1} - S^{-1} = S^{-1} \, (S - \hat S) \, \hat S^{-1}$ as well as $\lambda_{\min}(S) \geq \underline{\lambda}$, see Lemma \ref{lem:lowerboundeigen}, it follows that
\begin{align*}
    \big|\hat r^*(x) - r(x)\big|  \leq &\, \frac2{\underline{\lambda}}\,  \big\|\hat V - V\big\|_2 + \frac2{\underline{\lambda}^2} \, \big\| S - \hat S \big\|_{\mathrm{op}} \, \|V\|_2 \\ 
    & \, + a_m^{-1} \, \|\hat V\|_2\,    1\big( \lambda_{\min}(\hat S) \leq \underline{\lambda}/2\big)\, + \, 1\big(\lambda_{\min}(\hat S) \leq \underline{\lambda}/2\big) \, + \, 1\big(\lambda_{\min}(\hat S) \leq  \underline{\lambda}/2\big) \,\underline{\lambda}^{-1}\, \|V\|_2 \\
    & \, + \big| \bar r(x) - r(x)\big|.
\end{align*}
Passing to the square, upper-bounding the spectral norm with the Frobenius norm, and using independence of $\hat S$ and $\hat V$ we may bound
\begin{align}
   \E\big[ \big|\hat r^*(x) - r(x)\big|^2\big]  \leq &\,6\, \Big( \frac4{\underline{\lambda}^2}\, \E\big[ \big\|\hat V - V\big\|_2^2\big] + \frac4{\underline{\lambda}^4} \, \|V\|_2^2 \, \E\big[\big\| S - \hat S \big\|_F^2\big]  \nonumber \\ 
    & \, +     \Pr\big( \lambda_{\min}(\hat S) \leq \underline{\lambda}/2\big)\,\big( a_m^{-2} \, \E\big[\|\hat V\|_2^2\big]\, + \, 1  \, + \, \underline{\lambda}^{-2}\, \|V\|_2^2\big) \label{eq:boundingterms}\\
    & \, + \,\big| \bar r(x) - r(x)\big|^2\, \Big).\nonumber
\end{align}

Therefore, overall from \eqref{eq:boundingterms}, \eqref{eq:stochbound1} in Lemma \ref{lem:boundsterms}, Lemma \ref{lem:tsybakov2} and Proposition \ref{prop:bias} we get 
\eqref{eq:finitesamplebound}. 
\end{proof}
\begin{proof}[Proof of Corollary \ref{C:1}]
The choice of $h$ gives 
$$ \frac1{n h^d} + \frac1{m h^d} + h^{2 \beta} = {\cal O}\big(\min(m,n)^{-2\beta/(2\beta+d)}\big).$$
For the trimming term, the choice of $h $ implies $m \, h^d \gtrsim m^{ 2 \beta / ( 2\beta + d)}$, hence by the assumption $\log(1/a_m) = o\big(m^{2 \beta/ ( 2 \beta + d)} \big)$ we get
\begin{align*}
 a_m^{-2}\, \exp\big(- c \, m \, h^d \big) \,\big(1 + 1/(n h^d)\big) & \, \lesssim \exp\big(- c \, m \, h^d + 2\, \log(1/a_m)\big)\\
& \, = \, o\big(m^{-2 \beta/ ( 2 \beta + d)} \big) = o\big(\min(m,n)^{-2\beta/(2\beta+d)}\big),     
\end{align*}
which shows the first claim.  
Then the second claim follows by Fubini's theorem to exchange the expectation and the integral where the constant factors of the pointwise risk do not depend on $x \in {\cal G}$.
\end{proof}

\section*{Acknowledgements}

HH gratefully acknowledges financial support by the DFG, grant HO 3260/9-1. AM also gratefully acknowledges financial support by the DFG, Research Unit 5381, ME 2114/5-1. 

\printbibliography

\newpage

\appendix

\section{Remaining proofs for Section \ref{sec:rates}}\label{sec:proofremaingrates}

\begin{proof}[Proof of \eqref{eq:dervest}]
   Let $\theta_{{\bf k}}^* = \frac{(-h)^{|{\bf k}|}}{{\bf k}!} \, r^{{\bf k}}(x)$, $\bar \theta_{{\bf k}} = {\bf e}_{{\bf k}}^\top S^{-1} \, V$, $\hat \theta_{{\bf k}} = {\bf e}_{{\bf k}}^\top \hat S^{-1} \, \hat V$, so that replacing $\bar \alpha$ by $\bar \alpha_{{\bf k}} = S^{-1} {\bf e}_{{\bf k}}$ and repeating the argument in Proposition \ref{prop:bias} gives for $|\bar \theta_{{\bf k}} - \theta_{{\bf k}}^*|$ the same bound as in \eqref{eq:boundbiasfinal}. The bound \eqref{eq:boundingterms}  then applies to $\E\big[|\hat \theta_{{\bf k}} - \theta_{{\bf k}}^*|^2\big]$ as well with $|\bar \theta_{{\bf k}} - \theta_{{\bf k}}^*|^2$ replacing $\big| \bar r(x) - r(x)\big|^2$, and after rescaling we obtain for the estimator $\hat r_{{\bf k}}(x)$ that
 $$\mathbb{E}_{f,g} \big[\big| \hat r_{{\bf k}}^*(x) - r^{({\bf k})}(x)\big|^2\big] \, = \, {\cal O}\big(\min\{m,n\}^{-1}\, h^{-d-2\, |{\bf k}|}\, +  h^{2\,( \beta - |{\bf k}|)} \, + \, a_m^{-2} \, h^{- 2 |{\bf k}|}\, \exp(-c\, m\, h^d) \big( 1+ \frac1{n h^d}\big)\big),$$
 from which the claim follows. 
\end{proof}

Next we state two lemmas, which are proved in Section \ref{sec:proofslemmas}. 
%
%{lem:boundsterms}
\begin{lemma}\label{lem:boundsterms2}
 Under Assumption \ref{ass:boundedness}, for $h \in (0, \rho_1)$ and  $\eta >0$ we have that
\begin{align}\label{eq:tailbounds}
\begin{split}
   \Pr\big(\|\hat V - V \|_2 \geq \eta \big) & \, \leq \,  2\, (\# {\cal K})\, \exp\big(- \bar c_1\, n\, h^d\, \min(\eta,\eta^2)\big),\\
   \Pr\big(\|\hat S - S \|_{\mathrm{op}} \geq \eta \big) & \, \leq \,  2\, (\# {\cal K})^2 \, \exp\big(- \bar c_2\, m\, h^d\, \min(\eta,\eta^2)\big), 
 \end{split}
\end{align} 
for constants $ \bar c_i$, $i=1,2$, depending only on  $\rho_2$, $d$, $l$ and $K$.       
\end{lemma}

\begin{lemma}\label{lem:uniformbounds}
    Suppose that Assumption \ref{ass:boundedness} holds uniformly over ${\cal G}$, as in the function class ${\cal F}({\cal G})$ in \eqref{eq:funcclassG}, and that the kernel $K$ is Lipschitz continuous.  Then for $h \in (0, \rho_1)$ for which $\log(n)/n = o(h^d)$ and $\log(m)/m = o(h^d)$ we have that
    \begin{align}
    \begin{split}\label{eq:stochboundunif}
    \E\big[ \sup_{x \in {\cal G}} \big\|\hat V(x) - V(x)\big\|_2^2\big] & \leq C_1\, \frac{\log n}{n\, h^d} , \qquad 
    \E\big[\sup_{x \in {\cal G}}  \big\| S - \hat S \big\|_F^2\big] \leq C_2\, \frac{\log m}{m\, h^d} .
    \end{split}
    \end{align}
 Further, for $h \in (0, \rho_5)$ we have that
 \begin{equation}\label{eq:unifempeigbound}
    \Pr\big(\inf_{x \in {\cal G}}\, \lambda_{\min}(\hat S(x)) \leq \underline{\lambda}/2\big) \leq \, C_3\, h^{- d ( d+1)}\, \exp(-c\, m\, h^d).
\end{equation}    
\end{lemma}    
\begin{proof}[Proof of Theorem \ref{th:concinequ}]
Similarly as in the proof of Theorem \ref{T:1} we can write
\begin{equation}\label{eq:simplenontrimmed}
\hat r(x) - \bar r(x)  = {\bf e}_0^\top \, \hat S^{-1} \, \big(\hat V - V\big) +  \, {\bf e}_0^\top \, S^{-1} \, (S - \hat S) \, \hat S^{-1} \, V. 
\end{equation}

Therefore, using Lemmas \ref{lem:lowerboundeigen} and \ref{lem:tsybakov2} we obtain
\begin{align*}
& \, \Pr\Big(\big|\hat r(x) - \bar r(x)\big| \geq \delta \Big)\\ 
\, \leq \, & \Pr\big(\lambda_{\min}(\hat S) \leq \underline{\lambda}/2\big) +  \Pr\big(\lambda_{\min}(\hat S) > \underline{\lambda}/2, \,  \big|{\bf e}_0^\top \, \hat S^{-1} \, \big(\hat V - V\big) +  \, {\bf e}_0^\top \, S^{-1} \, (S - \hat S) \, \hat S^{-1} \, V\big| \geq \delta\big)\\
\, \leq \, & C\, \exp(-c\, m\, h^d) + \Pr\big(\|\hat V - V \|_2 \geq \underline{\lambda} \delta/4 \big) + \Pr\big(\|\hat S - S \|_{\mathrm{op}} \geq \underline{\lambda}^2 \delta/\big(4 \, \rho_2 \, \sqrt{\# {\cal K}}\big) \big) ,
\end{align*}
where for the last term we used the upper bound on $\|V \|_2$ from Lemma \ref{lem:boundsterms}. The proof is concluded by applying \eqref{eq:tailbounds}.  
\end{proof}

\begin{proof}[Proof of Theorem \ref{T:unifrate}]
    Suppose that $m$ is so large that $a_m \leq \underline{\lambda}/2$. Following the proof of Theorem \ref{T:1} leading up to \eqref{eq:boundingterms}, dropping $x$ from the notation we get
    \begin{align}
   \E\big[\sup_{{\cal G}} \big|\hat r^* - r\big|^2\big]  \leq &\,6\, \Big( \frac4{\underline{\lambda}^2}\, \E\big[ \sup_{{\cal G}} \big\|\hat V - V\big\|_2^2\big] + \frac4{\underline{\lambda}^4} \, \sup_{{\cal G}} \|V\|_2^2 \, \E\big[\sup_{{\cal G}} \big\| S - \hat S \big\|_F^2\big]  \nonumber \\ 
    & \, +     \Pr\big( \inf_{{\cal G}}\lambda_{\min}(\hat S) \leq \underline{\lambda}/2\big)\,\big( a_m^{-2} \, \E\big[\sup_{{\cal G}} \|\hat V\|_2^2\big]\, + \, 1  \, + \, \underline{\lambda}^{-2}\, \sup_{{\cal G}} \|V\|_2^2\big)\\
    & \, + \sup_{{\cal G}} \,\big| \bar r - r\big|^2\, \Big).\nonumber
\end{align}
    Then, since the bounds on the approximation error $\big| \bar r(x) - r(x)\big|$ from Proposition \ref{prop:bias} and on $\|V\|_2^2$ from Lemma \ref{lem:boundsterms} are uniform over $x \in {\cal G}$, it follows from Lemma \ref{lem:uniformbounds} that
    $$\E\big[\sup_{{\cal G}} \big|\hat r^* - r\big|^2\big] = {\cal O}\, \Big(   \frac{\log(1/h)}{\min(n,m) h^d}  + h^{- d\, (d+1)}\, a_m^{-2} \, \exp(-c\, m\, h^d) \, + h^{2 \beta}\,\Big),$$
    and the theorem follows by inserting the choice of $h$. 
\end{proof}

\section{Proofs for Section \ref{sec:asympnorm} on asymptotic normality}\label{sec:asympnormproofs}
We start with several lemmas before turning to the proofs of Theorem \ref{th:CLTgen} and of the results in Remark \ref{rem:derivest}. 
\begin{lemma}[Asymptotic normality]\label{lem:asympnormV}
     Impose Assumption \ref{ass:boundedness} and additionally assume  \eqref{eq:rpositive}. 
    \begin{enumerate}
        \item For the covariance matrix of $\hat V$, 
        $$ n\, h^d\, \cov\big( \hat V \big) = M - h^d\, V\, V^\top. $$
Further, setting $\bar \theta = S^{-1}\, V$, see Remark \ref{rem:localLS}, for the covariance matrix of $ \hat S \,  \bar \theta$, we have that
        $$ m\, h^d\, \cov\big( \hat S \,  \bar \theta \big) = N - h^d\, V\, V^\top. $$
        Here $M$ and $N$ are defined in \eqref{eq:defpopquantasympnormcov}.
        \item For some $\underline{\lambda}_M >0$ depending only on $\rho_3, \rho_4, l, d$ we have that $\lambda_{\min}(M) \geq \underline{\lambda}_M$ for all $h \in (0, \rho_5)$ and hence that $\lambda_{\min}(M - h^d\, V\, V^\top) \geq \underline{\lambda}_M/2$ for $h \leq \min\big(\rho_5, \underline{\lambda}_M/(2\,\# {\cal K}_l\, \rho_2^2)^{1/d}\big)$. 
        
        Similarly, for some $\underline{\lambda}_N >0$ depending on $\rho_2, \rho_3, \rho_4, d, l$ we have that $\lambda_{\min}(N) \geq \underline{\lambda}_N$ for $h \in (0, \min(\rho_5))$, 
        and hence  that $\lambda_{\min}(N - h^d\, V\, V^\top) \geq \underline{\lambda}_N/2$ for sufficiently small $h$. 
        \item Consider any sequence $u_{m,n} \in \R^{\# {\cal K}_l}$, $u_{m,n} \not=0$. 
        Then as $h \to 0$ and $\min(m,n)\, h^d \to \infty$ we have that
        \begin{equation}\label{eq:asmypnorm}
        \big(\min(m,n)\, h^d \big)^{1/2} \, \frac{u_{m,n}^\top \, \big(\hat V - V + (S - \hat S) \,  \bar \theta \big)}{\big(u_{m,n}^\top \, \Sigma_{m,n} \, u_{m,n} \big)^{1/2}}\, \stackrel{{\cal L}}{\to}\, {\cal N}(0, 1),
        \end{equation}
        where $\Sigma_{m,n}$ is defined in \eqref{eq:covmatrun}.
    \end{enumerate}
\end{lemma}
\begin{proof}[Proof of Lemma \ref{lem:asympnormV}]
    1.~: Straightforward computation

    2.~For $M$: Apply Lemma \ref{lem:Slowereigagain} with kernel $K^2$, degree $l$ and density $f$, which satisfies \eqref{eq:Boundary1}. 

    For $N$: Note that given $\alpha \in \R^{\# {\cal K}_l}$ we have that   
    \begin{equation}\label{eq:lowerboundminNstep1}
    \alpha^\top \, N  \, \alpha = \int g(x-h\,y) \, K^2(y) \, |\bar \theta^\top \, Q(y)\, \alpha^\top Q(y)|^2\,  \dd y.
    \end{equation}
    First apply Lemma \ref{lem:Slowereigagain} with kernel $K^2$, degree $2\, l$ and density $g$, resulting in the lower bound $\underline{\lambda}_2$. 
    The map $\Phi: \R^{\# {\cal K}_l} \times \R^{\# {\cal K}_l} \to \R^{\# {\cal K}_{2\,l}}$, which maps $(\theta,\alpha)$ to the coefficient vector of the polynomial $\theta^\top \, Q(y)\, \alpha^\top Q(y)$ of degree at most $2\, l$, is bilinear and in particular continuous, and $\Phi(\theta, \alpha) \not=0$ if both $\theta \not=0$ and $\alpha \not=0$. Hence by compactness and continuity, 
    $$ \kappa = \inf_{\|\theta\|=1,\, \|\alpha\|=1}\, \|\Phi(\theta, \alpha) \|>0,$$
    and bilinearity yields
    \begin{equation}\label{eq:blinearlowerbound}
        \|\Phi(\theta, \alpha) \| \, \geq \, \kappa \, \|\theta\| \, \|\alpha\|, \qquad \theta, \alpha \in \R^{\# {\cal K}_l}.
    \end{equation}
    Finally, from $V = S \, \bar \theta$ it follows that 
    $$ \|\bar \theta\|\, \geq \, \frac{\|V\|}{\|S\|_{\op}} \, \geq \, \frac{V_0}{\|S\|_{F}} \geq \frac{\underline{\lambda}_3}{\# {\cal K}\, \rho_2} >0$$
     where in the last inequality we used Lemma \ref{lem:boundsterms}, and to $V_0 = \int K(z)\, f(x - hz)\, \dd z$ we apply Lemma \ref{lem:Slowereigagain} with kernel $K$, density $f$ and degree $=0$ to yield $V_0 \geq \underline{\lambda}_3 >0$. 

    Overall, this together with  \eqref{eq:lowerboundminNstep1}, the choice of $\underline{\lambda}_2$ and  \eqref{eq:blinearlowerbound} yields for $h \in (0, \rho_5)$ that
    $$ \lambda_{\min}(N) \geq \kappa^2\, \|\bar \theta\|^2\, \underline{\lambda}_2.$$
    
    3. We intend to use the Lyapounov central limit theorem. Consider the pooled triangular array
    \begin{align*}
        \eta_{n,i} & \, = \, \frac{\big(\min(m,n)\, h^d \big)^{1/2}}{n}\, u_{m,n}^\top\, \Big(K_h(x-X_i)\, Q\big((x-X_i)/h\big) - V \Big)\, ,\qquad i=1, \ldots, n\, ,\\
                \eta_{n,n+j} & \, = \, - \, \frac{\big(\min(m,n)\, h^d \big)^{1/2}}{m}\, u_{m,n}^\top\, \Big(K_h(x-Y_j)\, Q\Big(\frac{x-Y_j}{h}\Big)\,Q\Big(\frac{x-Y_j}{h}\Big)^\top\, \bar \theta  - V \Big)\, ,\qquad j=1, \ldots, m\,,
    \end{align*}
    which are centered, row-wise independent random variables for which
    \begin{align}
        \sum_{i=1}^{n+m} \, \eta_{n,i} & \,=\, \big(\min(m,n)\, h^d \big)^{1/2}\, u_{m,n}^\top\,\big(\hat V - V + (S - \hat S) \,  \bar \theta \big), \nonumber\\
        \sum_{i=1}^{n+m} \, \var(\eta_{n,i}) & \,=\, u_{m,n}^\top\, \Sigma_{m,n}\, u_{m,n} \geq \|u_{m,n}\|^2\, \min(\underline{\lambda}_M,\underline{\lambda}_N)/2\,.\label{eq:lowerboundthevarlem}
    \end{align}
    To check the Lyapounov condition we upper-bound
    \begin{align*}
        \E\big[|\eta_{n,1}|^3 \big]\, & \leq \, \frac{\big(\min(m,n)\, h^d \big)^{3/2}}{n^3}\,  \|u_{m,n}\|^3\, \cdot 4\, \Big(\E \big\|K_h(x-X_1)\, Q\big((x-X_1)/h\big) \big\|^3 \ + \, \big\|V \big\|^3\Big)\\
        & \, \leq \, \frac{\big(\min(m,n)\, h^d \big)^{3/2}}{n^3}\,  \|u_{m,n}\|^3\, \big(h^{-2d}\, (\# {\cal K})^{3/2}\, \rho_2 + (\# {\cal K})^{3/2}\,\rho_2^3 \big)\\
        & \, \leq \, C\, \min(m,n)^{3/2}\, h^{-d/2}\, n^{-3}\, \|u_{m,n}\|^3\,
    \end{align*}
    and similarly 
    \begin{align*}
        \E\big[|\eta_{n,n+1}|^3 \big]\, & \, \leq \, \tilde C\, \min(m,n)^{3/2}\, h^{-d/2}\, m^{-3}\, \|u_{m,n}\|^3\,.
    \end{align*}
    Therefore together with \eqref{eq:lowerboundthevarlem},
    $$ \frac{\sum_{i=1}^{n+m} \, \E\big[|\eta_{n,i}|^3 \big]}{\big(\sum_{i=1}^{n+m} \, \var(\eta_{n,i})\big)^{3/2}} = \mathcal O\Big( \big( h^d\, \min(m,n)\big)^{-1/2}\Big) = o(1),$$
    that is, the Lyapounov condition. 
\end{proof}
\begin{remark}[Estimate of the covariance matrix]\label{rem:estcovclt}
    As estimators of $M$ and $N$ in \eqref{eq:defpopquantasympnormcov} consider
    \begin{align*}
    \hat M  & \,=\, \frac{1}{n}\sum_{i=1}^n (K^2)_h\big(x-X_i\big) \, Q\Big(\frac{x-X_i}{h}\Big)\,Q\Big(\frac{x-X_i}{h}\Big)^\top\,,\\  
    \hat N  & \,=\, \frac{1}{m}\sum_{j=1}^m (K^2)_h\big(x-Y_j\big) \, \Big|\hat\theta^\top \, Q\Big(\frac{x-Y_j}{h}\Big)\Big|^2\, Q\Big(\frac{x-Y_j}{h}\Big)\,Q\Big(\frac{x-Y_j}{h}\Big)^\top\,,
    \end{align*}
    with $\hat \theta = \hat S^{-1}\, \hat V$ as in Remark \ref{rem:localLS}, resulting in the estimator 
    $$\hat \Sigma_{m,n} = \frac{\min(m,n)}n \,(\hat M - h^d\, \hat V\, \hat V^\top) \, + \, \frac{\min(m,n)}m \, (\hat N - \, h^d\, \hat V\, \hat V^\top).$$
    of $\Sigma_{m,n}$.
\end{remark}
\begin{lemma}\label{lem:covmatconsist}
    Under Assumption \ref{ass:boundedness}, as $h \to 0$ and $\min(m,n)\, h^d \to \infty$ we have that
    \begin{equation*}
        \hat M - M = {\cal O}_\Pr \big((n\, h^d)^{-1/2}\big),\qquad \hat N - N = {\cal O}_\Pr \big((\min(m,n)\, h^d)^{-1/2}\big)
    \end{equation*}
    and hence
    \begin{equation}\label{eq:estcov}
         \hat \Sigma_{m,n} - \Sigma_{m,n} = {\cal O}_\Pr \big((\min(m,n)\, h^d)^{-1/2}\big).
    \end{equation}

\end{lemma}
    \begin{proof}[Proof of Lemma \ref{lem:covmatconsist}]
        The proof for $\hat M - M$ is similar to that for $\hat S - S$ in Lemma \ref{lem:boundsterms} and omitted. For the second claim write
        $$\hat N - N = \big(\hat N - \hat N (\bar \theta) \big) + \big(\hat N (\bar \theta) -  N \big),$$
        with $\hat N (\bar \theta)$ defined as $\hat N$ but with $\hat \theta$ replaced by $\bar \theta = S^{-1} V$. Now $\|\bar \theta\|$ is uniformly bounded for small $h$, see Lemmas \ref{lem:lowerboundeigen} and \ref{lem:boundsterms}, and $\hat N(\bar \theta) - N = {\cal O}_\Pr \big((m\, h^d)^{-1/2}\big)$ again follows as for $\hat S - S$. Furthermore, 
        $$ \|\hat N - \hat N (\bar \theta)\|_F \, = \, \big(\|\hat \theta - \bar \theta\|\, (\|\hat \theta \| + \|\bar \theta\|) \big)\, {\cal O}_\Pr (1) =  {\cal O}_\Pr\big((\min(m,n)\, h^d)^{-1/2} \big)$$
        upon using Lemma \ref{lem:boundsterms}, which then implies the claim for $\hat N$. Finally, since
        $$ \|\hat V\, \hat V^\top - V\, V^\top\|_{\op} \, \leq \, \big(\|\hat V\|_2 + \|V\|_2 \big)\,\big(\|\hat V\, - \,V\|_2 \big) \, = \, {\cal O}_\Pr\big((n\, h^d)^{-1/2} \big)$$
        by using Lemma \ref{lem:boundsterms}, \eqref{eq:estcov} also follows. 
    \end{proof}
    \begin{lemma}\label{lem:estempposdef}
       Given $\alpha \in \R^{\# {\cal K}_l}$ let
       \begin{align*}
           A_i & \, = \, K_h(x-X_i)\, \alpha^\top\,  Q\Big(\frac{x-X_i}{h}\Big)\, ,\\
           B_j & \, = \, K_h(x-Y_j)\, \alpha^\top\,  Q\Big(\frac{x-Y_j}{h}\Big) \, \hat \theta^\top\,  Q\Big(\frac{x-Y_j}{h}\Big)\, .
       \end{align*}
     Then
    \begin{equation}
         \bar A_n \, := \, \frac1n \, \sum_{i=1}^n A_i \, = \, \frac1m \, \sum_{j=1}^m B_j \, := \, \bar B_m  \, = \, \alpha^\top \hat V, 
     \end{equation}
    and for $\alpha = \hat S^{-1}\, {\bf e}_0$ these averages equal $\hat r(x)$. Further, 
    \begin{equation}
         \alpha^\top \, \hat \Sigma_{m,n} \, \alpha \, = \, \min(m,n)\, h^d\,\Big(\frac1{n^2}\, \sum_{i=1}^n \big(A_i - \bar A_n \big)^2 \, + \, \frac1{m^2}\, \sum_{j=1}^m \big(B_j - \bar B_m \big)^2 \Big), 
    \end{equation}     
    in particular, the matrix $\hat \Sigma_{m,n}$ is positive semidefinite. 
    \end{lemma}
This follows by direct computation.

\begin{proof}[Proof of Theorem \ref{th:CLTgen}]
 From \eqref{eq:simplenontrimmed}, on the event $\{\lambda_{\min}(\hat S) >0\}$, for which $\Pr(\lambda_{\min}(\hat S) >0) \to 1$,
\begin{align}
\hat r(x) - \bar r(x) & = {\bf e}_0^\top \, S^{-1} \, \big(\hat V - V\big) + {\bf e}_0^\top S^{-1} \, (S - \hat S) \, \hat S^{-1} \, \big(\hat V - V\big) \nonumber \\
& \, + \, \, {\bf e}_0^\top \, S^{-1} \, (S - \hat S) \,  S^{-1} \, V + \, \, {\bf e}_0^\top \, S^{-1} \, (S - \hat S) \, S^{-1} \, (S - \hat S) \, \hat S^{-1} \, \, V \nonumber \\ 
& = {\bf e}_0^\top \, S^{-1} \, \big(\hat V - V\big)  \, + \, {\bf e}_0^\top \, S^{-1} \, (S - \hat S) \,  S^{-1} \, V + {\cal O}_{\Pr}\Big(\frac1{\sqrt{m\,\min(m,n)}\, h^d} \Big),\label{eq:expandasymp}
\end{align}
The rate for the remainder term follows from Lemma \ref{lem:boundsterms}, and where
 $\|\hat S^{-1}\|_{\op} = 1/\lambda_{\min}(\hat S) = \mathcal O_{\Pr}(1)$ since $|\lambda_{\min}(\hat S) - \lambda_{\min}(S) | \leq \|\hat S - S\|_{\op} = o_{\Pr}(1)$ and $\lambda_{\min}(S) \geq \underline{\lambda} >0$ for all $h \in (0, \rho_5)$.

    Now, for the first statement \eqref{eq:CLTstand} of the theorem, consider the expansion  \eqref{eq:expandasymp}, and note that 
    $$ \sigma_{m,n}^2 \, = \, {\bf e}_0^\top \,\,S^{-1} \,\Sigma_{m,n} \, S^{-1}\, {\bf e}_0 \, \geq \, \|S^{-1}\, {\bf e}_0\|^2\, \min(\underline{\lambda}_M,\underline{\lambda}_N)/2 \, \geq \, \frac{\min(\underline{\lambda}_M,\underline{\lambda}_N)}{2\,(\# {\cal K})^2\, \rho_2^2 }$$
    using Lemma \ref{lem:asympnormV} and \eqref{eq:stochbound1}, so that the remainder in \eqref{eq:expandasymp} after normalization from \eqref{eq:CLTstand} is  still ${\cal O}_{\Pr}\big(1/(\sqrt{m\,h^d}) \big) = o_{\Pr}(1)$. Then apply Lemma \ref{lem:asympnormV}, 3.~with $u_{m,n} = S^{-1}\, {\bf e}_0$ to conclude  \eqref{eq:CLTstand}. 
    
    For the second statement, from Lemma \ref{lem:estempposdef} we obtain that 
    $$ \hat \sigma_{m,n}^2 \, = \, {\bf e}_0^\top \,\,\hat S^{-1} \,\hat \Sigma_{m,n} \, \hat S^{-1}\, {\bf e}_0.$$
    The second claim then follows 
from Slutsky's lemma, since  $\hat S^{-1} \,\hat \Sigma_{m,n} \, \hat S^{-1} - S^{-1} \,\Sigma_{m,n} \, S^{-1} = o_{\Pr}(1)$ , see Lemmas \ref{lem:covmatconsist} and \ref{lem:boundsterms},  and since the minimal eigenvalue of $S^{-1} \,\Sigma_{m,n} \, S^{-1}$ is bounded away from $0$. 
\end{proof}

Let's conclude the section by formulating a multivariate CLT for derivatives. Let $\cI \subseteq \cK_l$ and let $P_{\cI}$ denote the $\# \cK_l \times \# \cI$ matrix with columns $\bbe_{\bbk}$, $\bbk \in \cI$. Further set
\begin{equation}\label{eq:scalemultiderivative}
D_{\cI}\, = \, \text{diag}\big(\bbk! /(-h)^{|\bbk|}\big)_{\bbk \in \cI},\quad
\hat r_{\cI}(x) \, = \, D_{\cI}\, P_{\cI}^\top\, \hat\theta  = \big(\hat r_{\bbk}(x)\big)_{\bbk\in{\cI}},\quad
\bar r_{\cI}(x) \, = \, D_{\cI}\, P_{\cI}^\top\, \bar\theta  = \big(\bar r_{\bbk}(x)\big)_{\bbk\in{\cI}}    
\end{equation}
where as before $\hat \theta = \hat S^{-1} \hat V$, $\bar \theta =  S^{-1} V$, as well as 
\begin{equation*}
\Gamma_{\cI} \, = \, P_{\cI}^\top S^{-1}\, \Sigma_{m,n}\, S^{-1}\, P_{\cI}\,,\qquad
\hat\Gamma_{\cI} \, = \, P_{\cI}^\top \hat S^{-1}\, \hat \Sigma_{m,n}\, \hat S^{-1}\, P_{\cI} \, .   
\end{equation*}

\begin{theorem}[Multivariate derivative CLT]\label{th:CLTmultvarder}
Let $\cI \subseteq \cK_l$, set $q= \# \cI$, let $P_{\cI}$ denote the $\# \cK_l \times \# \cI$ matrix with columns $\bbe_{\bbk}$, $\bbk \in \cI$ and let $D_{\cI}$ as in \eqref{eq:scalemultiderivative}.     Impose Assumption \ref{ass:boundedness} and additionally assume  \eqref{eq:rpositive}. Then for sufficiently small $h$, $\Gamma_{\cal I}$ is invertible with minimal eigenvalue bounded away from zero, and as $h \to 0$ and $\min(m,n)\, h^d \to \infty$ we have that
    \begin{equation}\label{eq:CLTder}
      \big(\min(m,n)\, h^d\big)^{1/2}\ \Gamma_{\cal I}^{-1/2}\, D_{\cal I}^{-1}\,
      \big(\hat r_{\cal I}(x) - \bar r_{\cal I}(x)\big)
      \ \stackrel{{\cal L}}{\to}\ {\cal N}_q\big(0, \mathrm I_q\big).
    \end{equation}
    The statement remains true with $\Gamma_{\cI}$ replaced by $\hat\Gamma_{\cI}$.
\end{theorem}
\begin{proof}[Proof of Theorem \ref{th:CLTmultvarder}]
The proof is analogous to that of Theorem \ref{th:CLTgen}. Since $P_{\cI}^\top \, P_{\cI} = \mathrm{I}_q$, we have $\|P_{\cI}\, u\| = \|u\|$, $u \in \R^q$, and hence that
$$ u^\top\, \Gamma_{\cI}\, u \geq \lambda_{\min}\,(\Sigma_{m,n})\, \|S^{-1}\, P_{\cI}\, u\|^2 \geq \frac{\lambda_{\min}\,(\Sigma_{m,n})}{\|S\|_{\op}^2}\, \|u\|^2,$$
hence 
$$\lambda_{\min}\,(\Gamma_{\cI}) \geq \frac{\min(\underline{\lambda}_M,\underline{\lambda}_N)}{2\,(\# {\cal K})^2\, \rho_2^2 }.$$
Then, on $\{\lambda_{\min}(\hat S) >0\}$,
\begin{align*}
\Gamma_{\cal I}^{-1/2}\,D_{\cal I}^{-1}\, \big(\hat r_{\cI}(x) - \bar r_{\cI}(x)\big) 
& = \Gamma_{\cal I}^{-1/2}\, P_{\cI}^\top \, S^{-1} \, \Big(\big(\hat V - V\big)  \, + \, \, (S - \hat S) \,  S^{-1} \, V\Big) + {\cal O}_{\Pr}\Big(\frac1{\sqrt{m\,\min(m,n)}\, h^d} \Big).
\end{align*}
Then asymptotic normality follows from the Cram\'er-Wold theorem: For $u \in \R^q$ fixed, $u \not=0$,  set $u_{m,n}= S^{-1}\, P_{\cI}\, \Gamma_{\cI}^{-1/2}\, u$. Noting that $u_{m,n}^\top\, \Sigma_{m,n}\, u_{m,n} = u^\top\, \Gamma_{\cal I}^{-1/2}\, \Gamma_{\cal I}\, \Gamma_{\cal I}^{-1/2}\,u = \|u\|^2 \not=0$ is constant, we can apply Lemma \ref{lem:asympnormV}, 3., to conclude that proof of \eqref{eq:CLTder}.
    
The second statement then follows with the same argument as in the proof of Theorem \ref{th:CLTgen}. 
\end{proof}

\section{Proof of Theorem \ref{T:KL.1}} \label{sec:proofKL}
The Taylor approximation in \eqref{eq:KL.2} can be rigorously written as
\begin{equation}\label{eq:Taylorexact}
 \hat{\Phi} - \tilde{\Phi} \, = \, \big(\hat\Phi - \mathbb{E}\big[\hat\Phi \mid {\cal D}^{[1]}_{n^*,m^*}\big]\big) \, - \, \hat{R}^*_{n^*,m^*}\,,
\end{equation}
where the remainder term $\hat{R}^*_{n^*,m^*}$ satisfies
\begin{equation} \label{eq:Pr.KL.2} \big|\hat{R}^*_{n^*,m^*}\big| \, \leq \, \frac12 \, \big\|\Phi''\big\|_\infty\cdot \int \big|\hat{r}^*_{1,t}(x) \, - \, r(x)\big|^2 g(x)\, \dd x\,. \end{equation}
Throughout this proof all supremum norms are taken over the domain $[\rho_7,\rho_8]$. 
Then we derive that upon using the Cauchy-Schwarz w.r.t.~$g\, \dd \lambda$ that 
\begin{align} \nonumber
\mathbb{E} \big[\,\big|\hat{R}^*_{n^*,m^*}\big|^2\,\big] & \, \leq \, \frac14 \, \big\|\Phi''\big\|^2 \cdot \rho_2\cdot\int_{{\cal G}} \mathbb{E} \big[\min\big\{(\rho_8-\rho_7)^4 \, , \, \big|\hat{r}^*_1(x) \, - \, r(x)\big|^4\big\}\big] \dd x \\ \label{eq:Pr.KL.4}
& \, \lesssim \, \int_{{\cal G}} \mathbb{E} \big[\min\big\{(\rho_8-\rho_7)^4\, , \, \big|\hat{r}^*_1(x) - \bar{r}(x)\big|^4\big\}\big] \dd x \, + \, \int_{{\cal G}} \big|\bar{r}(x) - r(x)\big|^4 \dd x\,.
\end{align}
By Proposition \ref{prop:bias} the second term in (\ref{eq:Pr.KL.4}) has the asymptotic upper bound ${\cal O}(h^{4\beta})$. With respect to the first term we apply Theorem \ref{th:concinequ}, along with the remark thereafter, in order to obtain that
\begin{align}
\mathbb{E} \big[\min\big\{(\rho_8-\rho_7)^4\, , \, \big|\hat{r}^*_1(x) - \bar{r}(x)\big|^4\big\} \big]& \, = \, \int_{0}^{(\rho_8-\rho_7)} \, 4\, \delta^3\,\mathbb{P}\big(\big|\hat{r}^*_1(x) - \bar{r}(x)\big|\geq \, \delta\big)\, \dd  \delta \nonumber \\ 
& \, \lesssim \int_0^1\, 4\, \delta^3\, \exp\big(-\, c\, \min(m^*,n^*)\, h^d \, \delta^2\big)\, \dd \delta \label{eq:KL-tailbound}\\
& \, = \,  \int_0^1\, 2\, u\, \exp\big(-\, c\, \min(m^*,n^*)\, h^d \, u\big)\, \dd u \nonumber\\
& \, = \,  {\cal O} \Big(\min(m^*,n^*)^{-2}\, h^{-2\, d}\Big),\nonumber
\end{align}
where in \eqref{eq:KL-tailbound} we note that when applying \eqref{eq:expbound} from Theorem \ref{th:concinequ} we have on the compact interval $[0, \rho_8-\rho_7]$ that $\delta^2/(\rho_8-\rho_7) \leq \delta$ as well as $\delta^2 \leq (\rho_8-\rho_7)^2$. 
Therefore the imposed bandwidth selection provides that 
\begin{equation}\label{eq:KLremainder}
 \mathbb{E} \big[\big|\hat{R}^*_{n^*,m^*}\big|^2\big] \, = \, {\cal O}\big(\min\{n^*,m^*\}^{-4\beta/(2\beta+d)}\big) \, = \, o\big(\min\{n,m\}^{-1}\big)\,, 
 \end{equation}
what makes the above term asymptotically negligible.

Throughout the remaining proof we suppress the first argument in $\Psi_i$ and write $\Psi_i(x)$ for $\Psi_i(r;x)$. 
Introduce
$$ L_{m,n} \, = \, \frac1{m-m^*}\sum_{k=m^*+1}^m \big(\Psi_1(Y_k) - \E[\Psi_1(Y_1)]\big) \, + \, \frac1{n-n^*}\sum_{j=n^*+1}^n \big(\Psi_2(X_j) - \E[\Psi_2(X_1)]\big),$$
 set
$$\Delta_i(x)  \, = \, \hat{\Psi}_i(x) - \Psi_i(x), \qquad i=1,2,$$    
and note that for all $x\in {\cal G}$ we have that
\begin{align}\label{eq:boundsphi}
\begin{split}    
%\big|\Psi_1(x)\big| & \, \leq \, \|\Phi\|_\infty + \rho_8 \cdot\big\|\Phi'\big\|_\infty\,, \\
%\big|\Psi_2(x)\big| & \, \leq \, \big\|\Phi'\big\|_\infty\,, \\
\big|\Delta_1(x) \big| & \, \leq \, \rho_8\cdot \big\|\Phi''\big\|_\infty \cdot \big|\hat{r}^*_{1,t}(x) \, - \, r(x)\big|\,, \\
\big|\Delta_2(x)\big| & \, \leq \, \big\|\Phi''\big\|_\infty \cdot \big|\hat{r}^*_{1,t}(x) \, - \, r(x)\big|\,, 
\end{split}
\end{align}
where the first inequality follows from 
$$\big(\Phi(u) - \Phi'(u)\,u \big)' = - u\, \Phi''(u).$$
%the left side of (\ref{eq:Pr.KL.5}) converges to zero just by Corollary \ref{C:1}. Applying this result together with \eqref{eq:KLremainder} to \eqref{eq:Pr.KL.3} completes the proof of the first claim of the theorem.
%
%Focusing on the asymptotic distribution, note that the Taylor expansion with remainder can be written as 
% %
Now set
\begin{align*}
D_1(x) & \, = \, \Delta_1(x) - \E \big[\Delta_1(Y_m) \mid {\cal D}^{[1]}_{n^*,m^*} \big], \qquad D_2(x) \, = \, \Delta_2(x) - \E \big[\Delta_2(X_n) \mid {\cal D}^{[1]}_{n^*,m^*} \big].
\end{align*}
Then we have that
\begin{align*}
\, &\E\Big[\Big(\hat\Phi - \mathbb{E}\big[\hat\Phi \mid {\cal D}^{[1]}_{n^*,m^*}\big] - L_{m,n} \Big)^2 \Big]\\
\, = \, &  \E\Big[\Big(\frac1{m-m^*}\, \sum_{k=m^*+1}^{m}\, D_1(Y_k)\, + \, \frac1{n-n^*}\, \sum_{j=n^*+1}^{n}\, D_2(X_j) \Big)^2 \Big]\\ 
\, = \, & \E\Big[\mbox{var}\Big(\frac1{m-m^*}\, \sum_{k=m^*+1}^{m}\, \Delta_1(Y_k)\, + \, \frac1{n-n^*}\, \sum_{j=n^*+1}^{n}\, \Delta_2(X_j) \mid {\cal D}^{[1]}_{n^*,m^*} \Big) \Big]\\
\, = \, & \frac1{m-m^*}\,\E\big[\mbox{var}\big(\Delta_1(Y_m)\mid {\cal D}^{[1]}_{n^*,m^*}\big) \big]  + \, \frac1{n-n^*}\, \E\big[\mbox{var}\big(\Delta_2(X_n)\mid {\cal D}^{[1]}_{n^*,m^*}\big) \big]\\
\, \leq \, & \frac1{m-m^*}\,\E\big[\big(\Delta_1(Y_m)\big)^2 \big]  + \, \frac1{n-n^*}\, \E\big[\big(\Delta_2(X_n)\big)^2 \big] \\
\, = \, & o\big(n^{-1}\big)
\end{align*}
where the third equality follows from conditional independence given ${\cal D}^{[1]}_{n^*,m^*}$, and the final bound  
by using \eqref{eq:boundsphi}, Corollary \ref{C:1} and $n \asymp m$. 
Together with \eqref{eq:Taylorexact}, \eqref{eq:KLremainder}, the definition of $L_{m,n}$ and $m/n \to \kappa$ it follows that $\sqrt n\,(\hat\Phi - \tilde\Phi)$ and $\sqrt n\, L_{m,n}$ have the same limit in
$L_2$, while $n\,\E[L_{m,n}^2] \to {\cal V}(f,g)$ since $m^*/m \to 0$, $n^*/n \to 0$ and
$m/n \to \kappa$. The second claim then follows from 
$$ \sqrt n\, \big(\hat{\Phi} - \tilde{\Phi}\big) \, = \, \sqrt n\, L_{m,n} \, + \, o_{\Pr}(1)\,,$$
and the Lindeberg--L\'evy central limit theorem applied to each of the two independent averages. 

%Focusing on the asymptotic distribution we deduce that, for any $t\in \mathbb{R}$, 
%\begin{align*}
%\mathbb{E} & \exp\big\{it\sqrt{n} \big(\hat{\Phi} - \tilde{\Phi}\big)\big\} \\ & \, = \,  
%\mathbb{E}\, \mathbb{E}\Big( \exp\Big\{it\sqrt{n} \Big(\frac1{m-m^*} \sum_{k=m^*+1}^m \big(\hat{\Psi}_1(Y_k) - \mathbb{E}(\hat{\Psi}_1(Y_k)|{\cal D}^{[1]}_{n^*,m^*})\big) \\ & \hspace{3.4cm} + \, \frac1{n-n^*} \sum_{j=n^*+1}^n \big(\hat{\Psi}_2(X_j) - \mathbb{E}(\hat{\Psi}_2(X_j)|{\cal D}^{[1]}_{n^*,m^*})\big)\Big)\Big\}\Big|{\cal D}^{[1]}_{n^*,m^*}\Big) \, \pm \, o(1) \\
%& \, = \, \mathbb{E} \exp\Big\{-\frac12 t^2 \, \delta^{-1}\, \mbox{var}\big(\hat{\Psi}_1(Y_m)|{\cal D}^{[1]}_{n^*,m^*}\big) \, - \, \frac12 t^2 \, \mbox{var}\big(\hat{\Psi}_2(X_n)|{\cal D}^{[1]}_{n^*,m^*}\big)\Big\} \, \pm \, o(1) \\
%& \, = \, \exp\Big\{-\frac12 t^2 \, \delta^{-1}\, \mbox{var} \big(\Psi_1(Y_1)\big) \, - \, \frac12 t^2 \, \mbox{var}\big(\Psi_2(X_1)\big)\Big\} \, \pm \, o(1)\,,
%\end{align*}
%by dominated convergence. Again we use that pointwise convergence of the characteristic functions implies convergence in distribution so that the second claim of the theorem has been shown. 

\section{Proofs of technical lemmas}\label{sec:proofslemmas}

\begin{proof}[Proof of Lemma \ref{lem:Slowereigagain}]
For any $\alpha\in \mathbb{R}^{\#{\cal K}_l}$ consider
\begin{align*}   \alpha^\top S \alpha & %\, = \, \int K_h(x-y) \, \Big|\sum_{{\bf k}\in {\cal K}} \alpha_{\bf k} \cdot \prod_{j=1}^d\Big(\frac{x_j-y_j}h\Big)^{k_j}\Big|^2 \, g(y)\, dy \\ & 
\, = \, \int K(y) \, \big|\alpha^\top\, Q(y) \big|^2 \, g(x-y\,h)\, \dd y \\ 
\, & \, \geq \, \rho_3 \, \int_{G(x,h)} K(y) \, \big|\alpha^\top\, Q(y) \big|^2 \,  \dd y  \,, 
\end{align*}
using $K(z) > 0$ for all $\|z\|<1$ where 
\begin{align*} G(x,h) & \, := \, \big\{z\in \mathbb{R}^d \, : \, \|z\| \leq 1 \, , \, g(x-zh) \geq \rho_3\big\} \\
& \, = \, \big\{x/h - y/h \, : \, \|y-x\|\leq h \, , \, g(y) \geq \rho_3\big\}\,.
\end{align*}
By condition (\ref{eq:Boundary}), the shift-invariance and the $d$-homogeneity of $\lambda^d$ (with respect to stretch factors) we deduce that
$\lambda^d\big(G(x,h)\big) \geq \rho_4$ whenever $h\in (0,\rho_5)$. For any $\alpha \in \mathbb{R}^{\#{\cal K}_l}$ with $\|\alpha\|=1$ the $d$-variate polynomial $z \mapsto \alpha^\top\, Q(z)$, vanishes on a $\lambda^d$-null set only so that $\alpha^\top S \alpha > 0$ and, hence, the matrix $S$ is positive definite for $h\in (0,\rho_5)$. Moreover, for these $\alpha$ and $h$, the term $\alpha^\top S\alpha$ is bounded from below by
\begin{equation} \label{eq:infimum} \rho_3 \cdot \inf_{G\in {\cal G},P\in {\cal P}} \int_G K(z) \, |P(z)|^2 \, dz\,, \end{equation}
where ${\cal P}$ denotes the set of all $d$-variate polynomials with the degree $l$ whose coefficient vector has the Euclidean norm $1$ and ${\cal G}$ stands for the class of all Borel subsets $G$ of the closed Euclidean ball $B_1(0) =: B$ such that $\lambda^d(G)\geq \rho_4$.

Assume this infimum to be zero. Then there exist two sequences $(P_m)_m$ and $(G_m)_m$ in ${\cal P}$ and ${\cal G}$, respectively, such that
$$ \lim_{m\to\infty} \, \int_{G_m} K(z) \, |P_m(z)|^2 \, \dd z \, = \, 0\,.$$
Note that ${\cal P}_m$ forms a compact subset of a finite-dimensional linear space where all norms are equivalent. Therefore we have some subsequence $(P_{\sigma(m)})_m$ of $(P_m)_m$ which converges to some $\tilde{P}\in {\cal P}$ with respect to the supremum norm on the domain $B$. As $K$ is bounded it follows that
$$ \lim_{m\to\infty} \, \int_{G_{\sigma(m)}} K(z) \, |P(z)|^2 \, \dd z \, = \, 0\,, $$
where 
$$ \int_{G_{\sigma(m)}} K(z) \, |P(z)|^2 \, \dd z \, \geq \, \rho_4 \cdot \int K(z)\, |P(z)|^2 \, \dd \mu_{\sigma(m)}(z)\, \geq \, 0\,, $$
with the uniform distribution $\mu_{\sigma(m)}$ on $G_{\sigma(m)}$. Thus $\big(\mu_{\sigma(m)}\big)_m$ forms a tight sequence of probability measures supported on $B$ so that, by Prokhorov's theorem, some subsequence $\big(\mu_{\sigma'(m)}\big)_m$ of $\big(\mu_{\sigma(m)}\big)_m$ converges weakly to some probability measure $\mu$ that is supported on $B$. As the function $K\cdot P^2$ is continuous and supported on $B$ it follows that $\int K(z) \, |P(z)|^2 \, \dd \mu(z) \, = \, 0$ and, hence, that the support of $\mu$ is included in the intersection of the zero set of $P$ and $B$, which is denoted by ${\cal S}$. Note that ${\cal S}$ is a compact subset of $B$ with $\lambda^d({\cal S})=0$. We introduce the function
$$  d_k(z) \, := \, \max\big\{0 \, , \, 1 - k\cdot \inf\{\|z-s\| : s\in{\cal S}\}\big\}\,, \qquad z\in B, \, k\in \mathbb{N}\,. $$
As each function $d_k$ is continuous and bounded it holds that
$$ \lim_{m\to\infty} \int d_k(z)\, \dd\mu_{\sigma'(m)}(z) \, = \, \int d_k(z) \, \dd\mu (z) \, = \, 1\,, $$
for all $k\in \mathbb{N}$ since $d_k(z)=1$ for all $z\in {\cal S}$. On the other hand we have that
$$ \sup_{m\in \mathbb{N}}\, \int d_k(z)\, \dd\mu_{\sigma'(m)}(z) \, \leq \, \rho_4^{-1}\cdot \int_B d_k(z) \, \dd z\,, $$
for all $k\in \mathbb{N}$ where the right side of the above inequality converges to zero as $k\to\infty$ by dominated convergence since the functions $d_k$ take on their values in the interval $[0,1]$ and satisfy $\lim_{k\to\infty} d_k(z) \, = \, {\bf 1}_{{\cal S}}(z)$ for all $z\in B$. We arrive at a contradiction so that (\ref{eq:infimum}) has a lower bound which depends on $\rho_3$, $\rho_4$, $K$, $d$ and $\beta$ only. 
\end{proof}

\begin{proof}[Proof of Lemma \ref{lem:tsybakov2}]
We deduce this from the matrix Chernoff bound, \citet[Theorem 5.1.1, (5.1.5)]{tropp2015introduction}. Set
$$ U_h(y) = \Big( \Big(\frac{x-y}{h} \Big)^{{\bf k}}\Big)_{{\bf k} \in {\cal K}},\qquad A_j = K_h\big(x-Y_j \big)\,U_h(Y_j)\, U_h(Y_j)^\top, $$
so that 
$$m \hat S = \sum_{j=1}^m A_j.$$
The $A_1, \ldots, A_m$ are rank-one, positive semidefinite matrices, and since $\mathrm{supp}(K) \subseteq [-1,1]^d$ and $K$ is upper bounded by $1$, 
\begin{equation}\label{eq:boundsvalues}
|(A_j)_{{\bf k}, {\bf k}'}| = \Big|K_h(x-Y_j)\, \Big(\frac{x-Y_j}{h} \Big)^{{\bf k}+{\bf k}'} \Big| \leq h^{-d},\qquad {\bf k},{\bf k}' \in {\cal K}.
\end{equation}
For a $\ell \times \ell$-matrix $A$, $\|A\|_{\mathrm{op}} \leq \ell\, \max_{i,j} |A_{i,j}|$. Since $A_j$ is positive semidefinite, $\lambda_{\text{max}}(A_j) = \|A_j\|_{\mathrm{op}}$, so that 
$$ \lambda_{\text{max}}(A_j) = \|A_j\|_{\mathrm{op}} \leq \, (\# {\cal K})\, h^{-d}.$$
Further, $\E[\hat S] = S$ and $\lambda_{\min}(S) \geq \underline{\lambda}$ for $h \in (0, \rho_5)$ by Lemma \ref{lem:lowerboundeigen}. 
Therefore, using \citet[Theorem 5.1.1, (5.1.5)]{tropp2015introduction} with $\epsilon=1/2$ and observing that $\mathrm{e}^{-1/2}/\sqrt{1/2} = \sqrt{2/\mathrm{e}} < 1$ gives
\begin{align*}
\Pr\big(\lambda_{\min}(\hat S) \leq \underline{\lambda}/2\big) & \, \leq \, \Pr\big(\lambda_{\min}(m\, \hat S) \leq \, m\, \lambda_{\min}(S)/2\big)\\
 & \, \leq \, (\# {\cal K})\,\big(\mathrm{e}^{-1/2}/\sqrt{1/2}\big)^{\frac{m\, \lambda_{\min}(S)}{(\# {\cal K})\,h^{-d}}}\\
 & \, \leq \, (\# {\cal K})\,\exp(-c\, m\, h^d)\\
\end{align*}
for $c = \big(\underline{\lambda}/(2\, (\# {\cal K}) \big)\, \log(\mathrm{e}/2) $, proving the lemma. 
\end{proof}

\begin{proof}[Proof of Lemma \ref{lem:boundsterms}]
We use that the kernel $K$ satisfies $0 \leq K \leq 1$, and its support is contained in $B = B_1(0) \subseteq [-1,1]^d$. Hence since $h \in (0,\rho_1)$, if $z$ is in the support of $K$, by Assumption \ref{ass:boundedness} we have that $f(x-hz) \leq \rho_2$, and similarly for $g$. Now concerning \eqref{eq:stochbound1}, 
 \begin{align*} 
 \mathbb{E} \big[\|V - \hat{V}\|_2^2\big] & \, = \,  \sum_{{\bf k}\in {\cal K}} \mbox{var}\, \big(\hat{V}_{{\bf k}}\big) \, \leq \,  \sum_{{\bf k}\in {\cal K}} \,\frac1n\,  \E\Big[\Big({K}_h(x-X_1)\, \Big(\frac{x-X_{1}}h\Big)^{\bf k}\Big)^2\Big]\\
 & \, = \,  \frac1{n\, h^d} \, \sum_{{\bf k}\in {\cal K}}\, \int K^2(z) \,  z^{2\,{\bf k}} f(x-hz)\,\dd z\\ & \, \leq \,\#{\cal K}\,\cdot \rho_2 \cdot \frac1{n\, h^d}  \,, \\  
 \mathbb{E} \big[\|S - \hat{S}\|_F^2\big] & \, = \,  \sum_{{\bf k},{\bf k}'\in {\cal K}} \mbox{var}\, \big(\hat{S}_{{\bf k},{\bf k}'}\big) \, \leq \, (\# {\cal K})^2\, \rho_2 \,\cdot\frac1{m\, h^d} \,, \\
\|V \|_2^2 & \, = \,  \sum_{{\bf k}\in {\cal K}} \big(\E\, \big[\big|\hat{V}_{{\bf k}}\big|\big]\big)^2 \, \leq \, \sum_{{\bf k}\in {\cal K}}  \, \Big(\int \big|K(z) \, z^{{\bf k}} \big| \, f(x - hz)\, \dd z\Big)^2 \, \leq \,\# {\cal K} \cdot \rho_2^2 \,,   \\
 \mathbb{E} \big[\|\hat{V}\|_2^2\big] &\, \leq \,2\,  \mathbb{E} \big[\|V - \hat{V}\|_2^2\big] \, +\, 2\, \|V \|_2^2. 
 \end{align*}
\end{proof}
\begin{proof}[Proof of Lemma \ref{lem:boundsterms2}]
We prove \eqref{eq:tailbounds}, starting with 
 $\hat V$: For a coordinate ${\bf k} \in {\cal K}$ we can write
$$ \hat V_{{\bf k}} - V_{{\bf k}} = \frac1n\, \sum_{j=1}^n \big(\xi_{j,{\bf k}} - \E[\xi_{j,{\bf k}}] \big),\qquad \xi_{j,{\bf k}} = K_h(x - X_j)\, \Big( \frac{x-X_j}{h} \Big)^{{\bf k}}. $$
We intend to apply the scalar Bernstein inequality for the coordinates. To this end, note that
\begin{equation}\label{eq:boundsvalues1}
|\xi_{j,{\bf k}}| = \Big|K_h(x-X_j)\, \Big(\frac{x-X_j}{h} \Big)^{{\bf k}} \Big| \leq h^{-d},\qquad {\bf k} \in {\cal K}, \qquad \text{so} \quad |\xi_{j,{\bf k}} - \E[\xi_{j,{\bf k}}]| \leq 2\, h^{-d},
\end{equation}
and as above for \eqref{eq:stochbound1}, $\mbox{var}\, \big(\xi_{j,{\bf k}}\big) \leq \rho_2\, h^{-d}$. 
%
% \begin{align*} 
% \mbox{var}\, \big(\xi_{j,{\bf k}}\big) & \, \leq \, \E\big[\xi_{j,{\bf k}}^2 \big]  \, = \,  h^{-d} \,  \int K^2(z) \,  z^{2\,{\bf k}} f(x-hz)\,\dd z\\ 
% & \, \leq \, \rho_2 \,  h^{-d} \, \int K^2(z) \,  z^{2\,{\bf k}} \dd z  \, = \, C_V\,\,  h^{-d} \,. 
% \end{align*}
%
The Bernstein inequality yields
\begin{align}
    \Pr\big(|\hat V_{{\bf k}} - V_{{\bf k}}| \geq t \big) & \, \leq \, 2\, \exp\Big(- \frac{n\, h^d\, t^2/2}{\rho_2 + 2\, t/3} \Big)\nonumber \\
    & \, \leq \, 2\, \exp\Big(- \,\frac{n\, h^d}{2\, (\rho_2 + 1)} \,  \frac{t^2}{1 + t} \Big) \, \leq \, 2\, \exp\big(- \frac{n\, h^d\,  \min(t,t^2)}{4\,(\rho_2 + 1)}  \big) . \label{eq:boundcoordinate}
\end{align}
Since $\|\hat V - V \|_2 \, \leq \, (\#\, {\cal K})^{1/2}\, \max_{{\bf k} \in {\cal K}}|\hat V_{{\bf k}} - V_{{\bf k}}|$, 
we obtain from the union bound and \eqref{eq:boundcoordinate}, 
\begin{align*}
    \Pr\big( \|\hat V - V \|_2 \geq \eta \big)  \, \leq \, 2\, (\# \, {\cal K}) \, \exp\big(- \,n\, h^d\,  \min(\eta,\eta^2)/(4\,(\rho_2 + 1)\,\#\, {\cal K} ) \big), 
\end{align*}
which gives the first inequality in \eqref{eq:tailbounds}. 

For the second, passing to the Frobenius norm we get
$$\|\hat S - S \|_{\mathrm{op}} \, \leq \, \|\hat S - S \|_F \, \leq \, \#\, {\cal K}\, \max_{{\bf k}, {\bf k}' \in {\cal K}}|\hat S_{{\bf k}, {\bf k}'} - S_{{\bf k}, {\bf k}'}|,$$
and again, each term is bounded by $2\, h^{-d}$ with variance bounded by  $\rho_2\, h^{-d}$, and the  scalar Bernstein inequality and the union bound yield the conclusion. 
\end{proof}

\begin{proof}[Proof of Lemma \ref{lem:uniformbounds}]
    Let us start with the first part of \eqref{eq:stochboundunif}. From the Lipschitz continuity of the kernel $K$ and its definition it follows that $\hat V_{{\bf k}}$, ${\bf k} \in {\cal K}$, and hence its expected value $V_{{\bf k}}$ are  Lipschitz continuous with Lipschitz constant of order $h^{-(d+1)}$, independently of the sample, hence so are $\hat V$ and $V$. Denote their Lipschitz constants by $L_V\, h^{-(d+1)}$. 

    Given $\epsilon>0$ choose an $\epsilon$-cover ${\cal G}_{\epsilon}$ of  ${\cal G}$ with centers in ${\cal G}$ (internal cover) of cardinality $\# {\cal G}_{\epsilon} \leq C_{{\cal G}}\, \epsilon^{-d}$. Then 
    \begin{equation}\label{eq:discreteeff}
    \sup_{x \in {\cal G}} \big\|\hat V(x) - V(x)\big\|_2^2 \, \leq \, 2\, \sup_{x \in {\cal G}_\epsilon} \big\|\hat V(x) - V(x)\big\|_2^2 \, + \, 8\, L_V^2 \, h^{-2\, (d+1)}\, \epsilon^2, 
    \end{equation}
    and hence
    $$\E\Big[ \frac{n\, h^d}{\log n}\, \sup_{x \in {\cal G}} \big\|\hat V(x) - V(x)\big\|_2^2\Big] \, \leq \, 2\, \E\Big[ \frac{n\, h^d}{\log n}\,\sup_{x \in {\cal G}_\epsilon} \big\|\hat V(x) - V(x)\big\|_2^2 \Big] \, + \, 8\, L_V^2 \,\frac{n}{\log(n)\, h^{d+2}}\, \epsilon^2. $$

    Note also that $|\hat V_{{\bf k}}|  \leq h^{-d}$ as well as $|V_{{\bf k}}|\leq h^{-d}$. Hence setting 
    $A_{n,h} = \big(\# {\cal K}\, n/(h^d\, \log n\big)^{1/2}$
    for each $a>1$, from \eqref{eq:tailbounds} and the union bound it follows that
    \begin{align*}
    \, & \E\Big[ \frac{n\, h^d}{\log n}\,\sup_{x \in {\cal G}_\epsilon} \big\|\hat V(x) - V(x)\big\|_2^2 \Big] \\  \, \leq \, & a^2 \, +  \int_a^{2\, A_{n,h}}\, 2\, \eta \, \Pr\Big(\sup_{x \in {\cal G}_\epsilon} \big\|\hat V(x) - V(x)\big\|_2 \, \geq \, \big(\log(n) /(n\, h^d)\big)^{1/2} \, \eta \Big)\, \dd \eta\\ 
    \, \leq \, & a^2 \, + 8 \, (\# {\cal K})\, C_{{\cal G}}\, \epsilon^{-d}\, A_{n,h}^2\, \exp\big(- \bar c_1\, \log(n)\, a^2 \big)\,, 
    \end{align*}
    where also $a \leq \big((n\, h^d)/\log n\big)^{1/2}$ to remove the linear term $\delta$ in $\min(\delta, \delta^2)$ in \eqref{eq:tailbounds}, which is feasible from $\log(n)/n = o(h^d)$. Now take $\epsilon = (h^{d+2}/n)^{1/2}$ and $a$ so large such that $n^{1 + d/2}\, h^{- d (2 + d/2)} = {\cal O}\big(\exp(\bar c_1 \, a^2 \,\log n \big)$, again feasible from $\log(n)/n = o(h^d)$. This concludes the proof of the first part of \eqref{eq:stochboundunif}, the second being similar. 

    \smallskip

    Next we turn to \eqref{eq:unifempeigbound}, the uniform version of Lemma \ref{lem:tsybakov2}. We give a direct proof which builds on Weyl's inequality, without involving the matrix Chernoff bound as in the proof of Lemma \ref{lem:tsybakov2}. 
    Since $\inf_{x \in {\cal G}} \lambda_{\min}(S(x)) \geq \underline{\lambda} $ by Lemma \ref{lem:lowerboundeigen}, we have that 
    \begin{equation}\label{eq:boundinitial}
        \Pr\big(\inf_{x \in {\cal G}}\, \lambda_{\min}(\hat S(x)) \leq \underline{\lambda}/2\big) \, \leq \, \Pr\big(\sup_{x \in {\cal G}}\, \big|\lambda_{\min}(\hat S(x)) - \lambda_{\min}(S(x))\big| \, \geq \, \underline{\lambda}/2\big)
    \end{equation}
    From Weyl's inequality, first we get for $\eta >0$ that
\begin{align}
   \Pr\big(\big|\lambda_{\min}(\hat S(x)) - \lambda_{\min}(S(x))\big| \geq \eta \big) & \, \leq \, \Pr\big(\|\hat S(x) - S(x) \|_{\mathrm{op}} \geq \eta \big) \nonumber\\
   & \, \leq \,  2\, (\# {\cal K})^2 \, \exp\big(- \bar c_2\, m\, h^d\, \min(\eta,\eta^2)\big). \label{eq:exeigen}
\end{align}
where the second inequality follows from the second one in \eqref{eq:tailbounds}.  Further, 
\begin{align*}
    \big|\lambda_{\min}(\hat S(x)) - \lambda_{\min}(\hat S(y))\big| \, \leq \, \|\hat S(x) - \hat S(y) \|_{\mathrm{op}} \, \leq \, L\, h^{-(d+1)} \|x-y\| 
\end{align*}
the second inequality for a Lipschitz constant $L>0$, and the same inequality holds for $\big|\lambda_{\min}( S(x)) - \lambda_{\min}( S(y))\big|$. Therefore, analogously to \eqref{eq:discreteeff}, 
    $$ \sup_{x \in {\cal G}}\, \big|\lambda_{\min}(\hat S(x)) - \lambda_{\min}(S(x))\big| \leq \sup_{x \in {\cal G}_\epsilon}\, \big|\lambda_{\min}(\hat S(x)) - \lambda_{\min}(S(x))\big|\, + \,  2\, L \, h^{-\, (d+1)}\, \epsilon.$$
Setting $\epsilon = \underline{\lambda}\, h^{(d+1)}/(8\,L)$ we obtain
\begin{align*}
\Pr\big(\sup_{x \in {\cal G}}\, \big|\lambda_{\min}(\hat S(x)) - \lambda_{\min}(S(x))\big| \, \geq \, \underline{\lambda}/2\big) & \, \leq \, \Pr\big(\sup_{x \in {\cal G}_\epsilon}\, \big|\lambda_{\min}(\hat S(x)) - \lambda_{\min}(S(x))\big| \, \geq \, \underline{\lambda}/4\big)\\
& \, \leq \,  C_{{\cal G}}\, \underline{\lambda}^{-d}\, h^{-d\, (d+1)}\, (8\,L)^d\, 2\, (\# {\cal K})^2 \, \exp\big(- \bar c_2\, \underline{\lambda}^2 \, m\,  h^d\,/16\big),
\end{align*}
where in the last step we use the union bound and \eqref{eq:exeigen}. This together with \eqref{eq:boundinitial} implies \eqref{eq:unifempeigbound}. 
\end{proof}

%%%%%%%%%%%%%%%%%%%%%%%%%%%%%%%%%%%%%%%%%%%%%%%%%%%%%%%%%%%%%%%%%%%%%%
%%%%%%%%%%%%%%%%%%%%%%%%%%%%%%%%%%%%%%%%%%%%%%%%%%%%%%%%%%%%%%%%%%%%%%
%%%%%%%%%%%%%%%%%%%%%%%%%%%%%%%%%%%%%%%%%%%%%%%%%%%%%%%%%%%%%%%%%%%%%%

\end{document}